\documentclass{amsart}
\usepackage[pdftex,pdftitle={The Dehn function of
  SL(n;Z)},pdfauthor={Robert Young}]{hyperref}

\usepackage{amssymb}
\usepackage{graphicx}
\usepackage{color}

\newcommand{\cM}{\ensuremath{\mathcal{M}}}
\DeclareMathOperator{\diagmat}{diag}
\DeclareMathOperator{\SL}{SL}
\DeclareMathOperator{\SO}{SO}

\DeclareMathOperator{\Lip}{Lip}

\DeclareMathOperator{\Sol}{Sol}
\DeclareMathOperator{\diam}{diam}
\DeclareMathOperator{\area}{area}
\DeclareMathOperator{\trace}{trace}

\newcommand{\sqD}{\ensuremath{D^2}}
\newcommand{\F}{\ensuremath{\mathbb{F}}}
\newcommand{\D}{\ensuremath{\mathcal{D}}}
\newcommand{\cE}{\ensuremath{\mathcal{E}}}
\newcommand{\cEH}{{\ensuremath{{\mathcal{E}}_H}}}
\newcommand{\cP}{\ensuremath{\mathcal{P}}}
\newcommand{\cS}{\ensuremath{\mathcal{S}}}
\newcommand{\K}{\ensuremath{\mathbb{K}}}
\newcommand{\N}{\ensuremath{\mathbb{N}}}
\newcommand{\ellw}{\ell}
\newcommand{\ellc}{\ell}
\newcommand{\R}{\ensuremath{\mathbb{R}}}
\newcommand{\Q}{\ensuremath{\mathbb{Q}}}
\newcommand{\Z}{\ensuremath{\mathbb{Z}}}
\newcommand{\trdf}[1]{\delta^{\text{tri}}_{#1}}
\newcommand{\reldehn}[2]{\delta^{\text{rel}}_{#1\subset #2}}
\newcommand{\short}[1]{\hat{#1}}
\newcommand{\emptyword}{\varepsilon}
\newcommand{\blog}{\overline{\log}}

\newtheorem{thm}{Theorem}[section]

\newtheorem{lem}[thm]{Lemma}
\newtheorem{lemma}[thm]{Lemma}
\newtheorem{prop}[thm]{Proposition}
\newtheorem{cor}[thm]{Corollary}
\theoremstyle{remark}
\newtheorem{rem}[thm]{Remark}
\newtheorem{conj}[thm]{Conjecture}
\newcommand{\U}{\ensuremath{\mathcal{U}}}

\author{Robert Young}
\address{University of Toronto\\
  Dept.\ of Mathematics\\
  40 St. George Street, Rm.\ BA6290\\
  Toronto, ON  M5S 2E4\\
  Canada } \date{\today} \title{The Dehn function of $\SL(n;\Z)$}
\email{ryoung@math.toronto.edu}

\begin{document}

\begin{abstract}
  We prove that when $n\ge 5$, the Dehn function of $\SL(n;\Z)$ is
  quadratic.  The proof involves decomposing a disc in
  $\SL(n;\R)/\SO(n)$ into triangles of varying sizes.  By mapping
  these triangles into $\SL(n;\Z)$ and replacing large elementary
  matrices by ``shortcuts,'' we obtain words of a particular form, and
  we use combinatorial techniques to fill these loops.
\end{abstract}

\maketitle

\bibliographystyle{halpha}

\section{Introduction}

The Dehn function of a group is a geometric invariant that measures
the difficulty of reducing a word that represents the identity to the
trivial word.  Likewise, the Dehn function of a space measures the
difficulty of filling a closed curve in a space with a disc.  If a
group acts properly discontinuously, cocompactly, and by isometries on
a space, then the Dehn functions of the group and space grow at the
same rate.  Thus, for example, since a curve in the plane can be
filled by a disc of quadratic area, the Dehn function of $\R^2$ grows
like $n^2$ and the Dehn function of $\Z^2$, which acts on the plane,
also grows like $n^2$.

Dehn functions can grow very quickly.  For example, if $\Sol_3$ is the
space consisting of the 3-dimensional solvable Lie group
$$\Sol_3=\left\{\begin{pmatrix}
    e^t & 0 & x \\
    0 & e^{-t} & y \\
    0 & 0 & 1 \\
  \end{pmatrix}\middle| x,y,t\in \R\right\},$$
with a left-invariant metric, then its Dehn function grows
exponentially.  One reason for this is that $\Sol_3$ is isomorphic to
a horosphere in
the rank 2 symmetric space $H^2\times H^2$, where $H^2$ is the
hyperbolic plane.  This space contains 2-dimensional
flats which intersect $\Sol_3$ in large loops.  Since they are
contained in flats, the loops have fillings in $H^2\times H^2$ of
quadratic area, but since these fillings go far from $\Sol_3$, they
are exponentially difficult to fill in $\Sol_3$.

Subsets of symmetric spaces of rank at least $3$ often have smaller
Dehn functions.  For example, $(H^2)^3$ has a horosphere isometric to
the solvable group
$$\Sol_5=\left\{\begin{pmatrix}
    e^{t_1} & 0 & 0 & x \\
    0 & e^{t_2} & 0 & y \\
    0 & 0 & e^{t_3} & z \\
    0 & 0 & 0 & 1 \\
  \end{pmatrix}\middle| x,y,z,t_i\in \R, t_1+t_2+t_3=0\right\}.$$
As before, there are flats in $(H^2)^3$ which intersect $\Sol_5$, but
since $(H^2)^3$ has rank $3$, the intersections may be spheres instead
of loops.  Indeed, loops contained in unions of flats have fillings
contained in unions of flats and $\Sol_5$ has a quadratic Dehn
function.  This result was first stated by Gromov \cite[5.A$_9$]{GroAII}; a proof of a more
general case along the lines stated here was given by Dru\c{t}u
\cite{DrutuFilling}.

This suggests that the filling invariants of subsets of symmetric
spaces depend strongly on rank.  Some of the main test cases for this
idea are lattices acting on high-rank symmetric spaces.  If a lattice
acts on a symmetric space with non-compact quotient, one can remove an
infinite union of horoballs from the space to obtain a space on which
the lattice acts cocompactly.  When the symmetric space has rank 2,
removing these horoballs may create difficult-to-fill holes in flats,
as in $\Sol_3$, but when the rank is 3 or more, Gromov conjectured 
\begin{conj}[{\cite[5.D(5)(c)]{GroAII}}]\label{conj:mainConj}
  If $\Gamma$ is an irreducible lattice in a symmetric space with $\R$-rank at
  least 3, then $\Gamma$ has a polynomial
  Dehn function.
\end{conj}
(see also \cite{BestEskWort} for a more general conjecture which
generalizes the Lubotzky-Mozes-Raghunathan theorem.)  A special case
of this conjecture is the following conjecture of Thurston (see
\cite{GerstenSurv}):
\begin{conj}
  When $p\ge 4$, $\SL(p;\Z)$ has a quadratic Dehn function.
\end{conj}
In this paper, we will prove Thurston's conjecture when $p\ge 5$:
\begin{thm}\label{thm:mainThm} When $p\ge 5$, $\SL(p;\Z)$ has a
  quadratic Dehn function.
\end{thm}

When $p$ is small, the Dehn function of $\SL(p;\Z)$ is known; when
$p=2$, the group $\SL(2;\Z)$ is virtually free, and thus hyperbolic.
As a consequence, its Dehn function is linear.  When $p=3$, Epstein
and Thurston \cite[Ch.\ 10.4]{ECHLPT} proved that the Dehn function of
$\SL(3;\Z)$ grows exponentially; Leuzinger and Pittet generalized this
result to any non-cocompact lattice in a rank $2$ symmetric space
\cite{LeuzPitRk2}.  This exponential growth has applications to
finiteness properties of arithmetic groups as well; Bux and Wortman
\cite{BuxWortman} describe a way that the constructions in
\cite{ECHLPT} lead to a proof that $\SL(3;\F_q[t])$ is not finitely
presented (this fact was first proved by Behr \cite{Behr77}), then generalize to a large class of lattices in reductive
groups over function fields.  The previous best known bound for the
Dehn function of $\SL(p;\Z)$ when $p\ge 4$ was exponential; this
result is due to Gromov, who sketched a proof that the Dehn function
of $\Gamma$ is bounded above by an exponential function
\cite[5.A$_7$]{GroAII}.  A full proof of this fact was given by
Leuzinger \cite{LeuzingerPolyRet}.

Notable progress toward Conj.~\ref{conj:mainConj} was made by
Dru\c{t}u \cite{DrutuFilling} in the case that $\Gamma$ is a lattice
in $G$ with $\Q$-rank 1.  In this case, $\Gamma$ acts cocompactly on a
subset of $G$ constructed by removng infinitely many disjoint
horoballs.  Dru\c{t}u showed that if $G$ has $\R$-rank 3 or greater,
then the boundaries of these horoballs satisfy a quadratic filling
inequality and that if $\Gamma$ has $\Q$-rank 1, then it enjoys an
``asymptotically quadratic'' Dehn function, i.e., its Dehn function is
bounded by $n^{2+\epsilon}$ for any $\epsilon>0$.  More recently,
Bestvina, Eskin, and Wortman \cite{BestEskWort} have made progress
toward a higher-dimensional generalization of
Conj.~\ref{conj:mainConj} by proving filling estimates for
$S$-arithmetic lattices.

The basic idea of the proof of Theorem~\ref{thm:mainThm} (we will give
a more detailed sketch in Sec.~\ref{sec:sketchProof}) is to use
fillings of curves in the symmetric space $\SL(p;\R)/\SO(p)$ as
templates for fillings of words in $\SL(p;\Z)$.  Fillings which lie in
the thick part of $\SL(p;\R)/\SO(p)$ correspond directly to fillings
in $\SL(p;\Z)$, but in general, an optimal filling of a curve in the
thick part may have to go deep into the cusp of $\SL(p;\Z)\backslash
\SL(p;\R)/\SO(p)$.  Regions of this cusp correspond to parabolic
subgroups of $\SL(p;\Z)$, so we develop geometric
techniques to cut the filling into pieces which each lie in one such
region.  This reduces the problem of filling the original word to the
problem of filling words in parabolic subgroups of $\Gamma$.  This
step is fairly general, and these geometric techniques may be applied
to a variety of groups.  We fill these words using combinatorial
techniques, especially the fact that $\Gamma$ contains many
overlapping solvable subgroups.  This step is specific to $\SL(p;\Z)$
and is the step that fails in the case $p=4$.

In Section~\ref{sec:prelims}, we define some of the notation and
concepts that will be used in the rest of the paper.  Readers who are
already familiar with Dehn functions may wish to skip parts of this
section, but note that Sec.~\ref{subsec:slpr} introduces much of the
notation we will use to describe subgroups and elements of $\SL(p)$,
and that Sec.~\ref{subsec:templates} introduces the new notion of
``templates'' for fillings.  

In Section~\ref{sec:sketchProof}, we give an outline of the proof of
Theorem~\ref{thm:mainThm}.  This outline reduces
Theorem~\ref{thm:mainThm} to a series of lemmas which decompose words
in $\SL(p;\Z)$ into words in smaller and smaller subgroups of
$\SL(p;\Z)$.  In Section~\ref{sec:redParaProof} we describe the main
geometric technique: a method for decomposing words in $\SL(p;\Z)$
into words in maximal parabolic subgroups.  Then, in
Sections~\ref{sec:normalForms} and \ref{sec:shortManip}, we describe a
normal form for elements of $\SL(p;\Z)$ and prove several
combinatorial lemmas giving ways to manipulate this normal form.
Finally, in Sections~\ref{sec:redDiagProof} and
\ref{sec:baseCaseProof}, we apply these techniques to prove the
lemmas, and in Section~\ref{sec:open}, we ask some open questions.

% Explicitly state how much is necessary for polynomial proof?

Some of the ideas in this work were inspired by discussions at the
American Institute of Mathematics workshop, ``The Isoperimetric
Inequality for $\SL(n;\Z)$,'' and the author would like to thank the
organizers, Nathan Broaddus, Tim Riley, and Kevin Wortman; and
participants, especially Mladen Bestvina, Alex Eskin, Martin Kassabov,
and Christophe Pittet.  The author would also like to thank Tim Riley,
Yves de Cornulier, Kevin Wortman, and Mladen Bestvina for many helpful
conversations while the author was visiting Bristol University,
Universit\'e de Rennes, and University of Utah.  Parts of this paper
were completed while the author was a visitor at the Institut des
Hautes \'Etudes Scientifiques and a Courant Instructor at New York University.

\section{Preliminaries}\label{sec:prelims}
In this section, we will describe some of the concepts and notation we
will use throughout this paper.

We use a variant of big-O notation throughout this paper; the notation
$$f(x,y,\dots)=O(g(x,y,\dots))$$ means that there is a $c>0$
such that $|f(x,y,\dots)|\le c g(x,y,\dots)+c$ for all values of the
parameters.  In most cases, $c$ will also depend implicitly on $p$.

If $f:X\to Y$ is Lipschitz, we say that $f$ is $c$-Lipschitz if
$d(f(x),f(y))\le cd(x,y)$ for all $x,y\in X$, and we let $\Lip(f)$ be
the infimal $c$ such that $f$ is $c$-Lipschitz.

% other general little things

\subsection{Words and curves}
If $G$ is a group with finite generating set $S$, we call a formal
product of elements of $S$ and their inverses a {\em word} in $S$.  By
abuse of notation, we will also call this a word in $G$ and leave $S$
implicit.  We denote the empty word by $\emptyword$.  There is a
natural evaluation map taking words in $G$ to $G$, and we say that a
word {\em represents} its corresponding group element.  If
$w=s_1^{\pm1}\dots s_n^{\pm1}$, we say that $w$ has length $\ell(w)=n$.

If $G$ acts on a space $X$, words in $G$ correspond to
curves in $X$.  Let $X$ be a connected simplicial
complex or riemannian manifold and let $G$ act on $X$ by maps of
simplicial complexes or by isometries, respectively.  Let $S$ be a
finite generating set for $G$ and for all $s\in S$, let
$\gamma_s:[0,1]\to X$ be a curve connecting $x_0$ to $sx_0$.  Let
$\gamma_s^{-1}$ be the same curve with the reverse parameterization.
If $w=s_1^{\pm1}\dots s_n^{\pm1}$ is a word in $S$ which represents
$g$, we can construct a curve $\gamma_w$ in $X$ by concatenating
translates of the $\gamma_{s_i}^{\pm1}$'s.  The resulting
curve connects $x_0$ and $gx_0$ and its length is bounded by the
length of $w$:
$$\ell(\gamma_w)\le \ell(w) \max_{s\in S} \ell(\gamma_s).$$

\subsection{Dehn functions and the Filling Theorem}

A full introduction to the Dehn function can be found in \cite{Bridson}.
We will just summarize some necessary results and notation here.  
If
$$G=\langle h_1,\dots,h_d \mid r_1,\dots,r_s\rangle$$ 
is a finitely presented group and $w$ is a word representing the
identity, there is a sequence of steps which reduces $w$ to the empty word,
where each step is a free reduction or insertion or the application of
a relator.  We call the number of applications of
relators in a sequence its {\em cost}, and we call the minimum cost of
a sequence which reduces $w$ to $\emptyword$ the {\em
  filling area} of $w$, denoted by $\delta_G(w)$.  We then define the
{\em Dehn function} of $G$ to be
$$\delta_G(n)=\max_{\ellw(w)\le n} \delta_G(w),$$
where the maximum is taken over words representing the identity.  For
convenience, if $v,w$ are two words representing the same element of
$H$, we define $\delta_G(v,w)=\delta_G(vw^{-1})$; this is the
minimum cost to transform $v$ to $w$.

Likewise, if $X$ is a simply-connected riemannian manifold or
simplicial complex (more generally a local Lipschitz neighborhood
retract) and $\gamma:S^1\to X$ is a Lipschitz closed curve, we define
its filling area $\delta_X(\gamma)$ to be the infimal area of a
Lipschitz map $D^2\to X$ which extends $\gamma$.  We can define the
Dehn function of $X$ to be
$$\delta_X(n)=\sup_{\ellc(\gamma)\le n} \delta_X(\gamma),$$
where the supremum is taken over null-homotopic closed curves.  As in
the combinatorial case, if $\beta$ and $\gamma$ are two curves
connecting the same points and which are homotopic with their
endpoints fixed, we define $\delta_X(\beta,\gamma)$ to be the infimal
area of a homotopy between $\beta$ and $\gamma$ which fixes their
endpoints.

Note that combinatorial fillings can be converted into geometric
fillings.  Gromov stated the following theorem connecting geometric and
combinatorial Dehn functions, a proof of which can be found in
\cite{Bridson}.
\begin{thm}[Gromov's Filling Theorem]\label{thm:GroFill}
  If $X$ is a simply connected riemannian manifold or simplicial
  complex and $G$ is a finitely presented group acting properly
  discontinuously, cocompactly, and by isometries on $M$, then
  $\delta_G\sim \delta_M$.
\end{thm}
Here, $f \sim g$ if $f$ and $g$ grow at the same rate.  Specifically,
if $f,g:\N\to \N$, let $f\lesssim g$ if and only if there is a $c$
such that
$$f(n)\le c g(cn+c)+c\text{ for all }n$$
and $f\sim g$ if and only if $f\lesssim g$ and $g\lesssim f$.
%quibbles: Occasionally we mix these defs in "nonkosher" ways.

\subsection{$\SL(p;\R)$ and $\SL(p;\Z)$}\label{subsec:slpr}

Let $\Gamma=\SL(p;\Z)$ and let $G=\SL(p)=\SL(p;\R)$.  One of the main
geometric features of $G$ is that it acts on a non-positively curved
symmetric space which we denote by $\cE$.  Let $\cE=\SL(p;\R)/\SO(p)$.
We consider $\cE$ with the metric obtained from the inner product
$\langle u,v\rangle=\trace(u^{tr}v)$ on the space of symmetric
matrices.  Under this metric, $\cE$ is a non-positively curved
symmetric space.  The lattice $\Gamma$ acts on $\cE$ with finite
covolume, but the action is not cocompact.  Let $\cM:=\Gamma\backslash
\cE$.  If $x\in G$, we write the equivalence class of $x$ in $\cE$ as
$[x]_\cE$; similarly, if $x\in G$ or $x\in \cE$, we write the
equivalence class of $x$ in $\cM$ as $[x]_\cM$.

If $g\in G$ is a matrix with coefficients $\{g_{ij}\}$, we define
$$\|g\|_2=\sqrt{\sum_{i,j}g_{ij}^2},$$
$$\|g\|_\infty=\max_{i,j}|g_{ij}|.$$
Note that for all $g,h\in G$, we have $\log \|g\|_2 = O(d_G(I,g)).$

One key fact about the geometry of $\SL(p;\Z)$ is a theorem of
Lubotzky, Mozes, and Raghunathan \cite{LMRComptes}:
\begin{thm}\label{thm:LMR}
  The word metric on $\SL(p;\Z)$ for $p\ge 3$ is equivalent to the
  restriction of the riemannian metric of $\SL(p;\R)$ to $\SL(p;\Z)$.
  That is, there is a $c$ such that for all $g\in \SL(p;\Z)$, we have
  $$c^{-1} d_G(I,g)\le d_{\Gamma}(I,g)\le c d_G(I,g).$$
\end{thm}

We define subset of $G$ on which $\Gamma$ acts cocompactly by
interpreting $\cE$ as the set of unimodular bases of $\R^p$ up to
rotation; if $v_1,\dots,v_n$ are the rows of a matrix $g$, then the
point $[g]_\cE$ corresponds to the basis $\{v_1,\dots,v_n\}$.  An
element of $\Gamma$ acts on $\cE$ by replacing the basis elements by
integer combinations of basis elements.  This preserves the lattice
that they generate, so we can think of $\cM$ as the set of
unit-covolume lattices in $\R^p$ up to rotation.  Nearby points in
$\cM$ or $\cE$ correspond to bases or lattices which can be taken into
each other by small linear deformations of $\R^p$.  Note that this set
is not compact -- for instance, the injectivity radius of a lattice is
a positive continuous function on $\cM$, and there are lattices with
arbitrarily small injectivity radiuses.

Let $\cE(\epsilon)$ be the set of points which correspond to lattices
with injectivity radius at least $\epsilon$.  This is invariant under
$\Gamma$, and when $0<\epsilon \le 1/2$, it is contractible and
$\Gamma$ acts on it cocompactly \cite{ECHLPT}.  We call
$\cE(\epsilon)$ the {\em thick part} of $\cE$, and its preimage
$G(\epsilon)$ in $G$ the thick part of $G$.  ``Thick'' here refers to
the fact that the quotients $\Gamma\backslash \cE(\epsilon)$ and
$\Gamma\backslash G(\epsilon)$ have injectivity radius bounded below.

Epstein, et al. construct a Lipschitz deformation retraction from
$\cE$ to $\cE(\epsilon)$, so $\cE(\epsilon)$ is a local Lipschitz
neighborhood retract in $\cE$.  % reference Groft?
The results of \cite{GroftA} imply that Gromov's Filling Theorem
extends to such retracts, so proving a filling inequality for $\Gamma$
is equivalent to proving one for $\cE(\epsilon)$.

We will also define some subgroups of $G$.  In the following, $\K$
represents either $\Z$ or $\R$.  Let $z_1,\dots,z_p$ be the standard
generators for $\Z^p$, and if $S\subset \{1,\dots,p\}$, let
$\R^S=\langle z_s\rangle_{s\in S}$ be a subspace of $\R^p$.  If $q\le
p$, there are many ways to include $\SL(q)$ in $\SL(p)$.  Let $\SL(S)$
be the copy of $\SL(\#S)$ in $\SL(p)$ which acts on $\R^S$ and fixes
$z_t$ for $t\not \in S$.  If $S_1,\dots,S_n$ are disjoint subsets of
$\{1,\dots, p\}$ and $S:=\bigcup S_i$, let
$$U(S_1,\dots,S_n)\subset \SL(S;\Z)$$
be the subgroup of matrices preserving the flag
$$\R^{S_i}\subset \R^{S_i\cup S_{i-1}} \subset \dots\subset \R^S $$
when acting on the right.  If the $S_i$ are sets of consecutive
integers in increasing order, $U(S_1,\dots,S_n)$ is block upper-triangular.  For example, $U(\{1\},\{2,3\})$ is the subgroup of
$\SL(3;\K)$ consisting of matrices of the form:
$$\begin{pmatrix}
  * & *  & * \\
  0 & *  & * \\
  0 & *  & *
\end{pmatrix}.$$ If $d_1,\dots d_n>0$, let $U(d_1,\dots,d_n)$ be
the group of upper block triangular matrices with blocks of the given
lengths, so that the subgroup illustrated above is $U(1,2)$.  If
$\sum_i d_i=p$, this is a parabolic subgroup of $\Gamma$; if $\sum_i
d_i<p$ then it is a parabolic subgroup of $\SL(\sum_i d_i;\Z)$.  Let
$\cP$ be the set of groups $U(d_1,\dots,d_n)$ with $\sum_i d_i=p$,
including $U(p)=\Gamma$.  Any parabolic subgroup of $\Gamma$ is
conjugate to a unique such group.

One feature of $\SL(p;\Z)$ is that it has a particularly simple
presentation, the Steinberg presentation.
If $1\le i\ne j\le p$, let $e_{ij}(x)\in \Gamma$ be the matrix which
consists of the identity matrix with the $(i,j)$-entry replaced by
$x$; we call these {\em elementary matrices}.  Let
$e_{ij}:=e_{ij}(1)$.  When $p\ge 3$, there is a finite presentation which has the
matrices $e_{ij}$ as generators \cite{Steinberg, Milnor}:
\label{sec:origSteinberg}
\begin{align}
  \notag  \Gamma=\langle e_{ij} \mid \; &[e_{ij},e_{kl}]=I & \text{if $i\ne l$ and $j\ne k$}\\
  & [e_{ij},e_{jk}]=e_{ik} & \text{if $i\ne k$}\label{eq:steinberg}\\
  \notag  & (e_{ij} e_{ji}^{-1} e_{ij})^4=I \rangle,
\end{align}
where we adopt the convention that $[x,y]=xyx^{-1}y^{-1}$.  
We will use a slightly expanded set of generators.  Let
$$\Sigma:=\{e_{ij}\mid 1\le i\ne j\le p\}\cup D,$$
where $D\subset \Gamma$ is the set of diagonal matrices in
$\SL(p;\Z)$; note that this set is finite.  If $R$ is the set of relators given
above with additional relations expressing each element of $D$ as a
product of elementary matrices, then $\langle \Sigma\mid R\rangle$ is
a finite presentation of $\Gamma$ with relations $R$.  Furthermore, if
$H=\SL(q;\Z)\subset \SL(p;\Z)$ or if $H$ is a subgroup of
block-upper-triangular matrices, then $H$ is generated by $\Sigma \cap
H$.

\subsection{Templates and relative Dehn functions}\label{subsec:templates}

In this section, we introduce some new definitions which we will use
in the proof of Thm.~\ref{thm:mainThm}.

The \emph{relative Dehn function}, $\reldehn{H}{G}$, of a subgroup
$H\subset G$ describes the difficulty of filling words in $H$ by discs
in $G$.  If $G=\{S|R\}$ is a finite presentation for $G$ and
$S_0\subset S$ is a generating set for $H$ and $S_0^*$ represents the
set of words in $S_0$, we define
$$\reldehn{H}{G}(n)=\max_{w\in S_0^*,\ellw(w)\le n} \delta_G(w).$$
By definition, $\reldehn{G}{G}(n)=\delta_G(n).$

If $\omega:G\to S^*$ is a map such that for all $g$, $\omega(g)$ is a
word representing $g$ which has length $\sim \ellw(g)$, we say that
$\omega$ is a \emph{normal form} for $G$.  We will define a
\emph{triangular relative Dehn function}, $\trdf{H,\omega}$, which
describes the difficulty of filling ``$\omega$-triangles'' with
vertices in $H$.  If $g_1,g_2,g_3\in G$, we say that
$$\Delta_\omega(g_1,g_2,g_3)=\omega(g_1^{-1}g_2)\omega(g_2^{-1}g_3)\omega(g_3^{-1}g_1)$$
is the \emph{$\omega$-triangle} with vertices $g_1, g_2, g_3$.  Then we can
define
$$\trdf{H,\omega}(n)=\mathop{\max_{h_1, h_2, h_3\in
    H}}_{\diam\{h_1, h_2, h_3\}\le n}
\delta_G\left(\Delta_\omega(h_1,h_2,h_3)\right).$$

If $h\in H$, then even though $\omega(h)$ will have endpoints in
$H$, it need not be a word in $H$, so upper bounds on $\reldehn{H}{G}$
might not lead to upper bounds on $\trdf{H,\omega}$.  On the other
hand, we can bound $\reldehn{H}{G}$ by decomposing words in $H$ into
$\omega$-triangles.  We can describe these decompositions using
\emph{templates}.  Let $\tau$ be a triangulation of $D^2$ whose
vertices are labeled by elements of $G$; this is a template.  If the
boundary vertices of $\tau$ are labeled (in order), $g_1,\dots, g_n$,
we let
$$w_\tau=\omega(g_1^{-1}g_2)\dots\omega(g_{n-1}^{-1}g_n)\omega(g_{n}^{-1}g_1),$$
and call $w_\tau$ the {\em boundary word} of $\tau$.  If
$w=w_1\dots w_n$ is a word and the boundary of $\tau$ is an $n$-gon
with labels $I$, $w_1$, $w_1w_2$, $\dots$, $w_1\dots w_{n-1}$, we call
$\tau$ a {\em template for $w$}.

\begin{figure}
  \def\svgwidth{4in}
  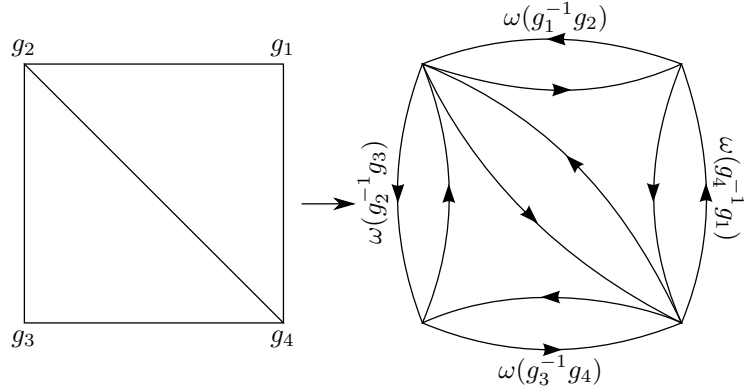
  \caption{\label{fig:template} The boundary word of the template on
    the left is
    $w_\tau=\omega(g_1^{-1}g_2)\omega(g_2^{-1}g_3)\omega(g_3^{-1}g_4)\omega(g_4^{-1}g_1)$.
    On the right, we use the template to break $w_\tau$ into five
    $\omega$-bigons of the form
    $\omega(g_i^{-1}g_j)\omega(g_j^{-1}g_i)$ and two
    $\omega$-triangles of the form $\Delta_\omega(g_i,g_j,g_k)$.}
\end{figure}

We say that we can \emph{break} a word $w$ into some words $w_i$ at
cost $C$ if each of the $w_i$'s represent the identity and there exist
words $g_i$ such that
$$\delta_\Gamma(w,\prod_i g_i w_i g_i^{-1})=C.$$
In particular, this means that 
$$\delta_\Gamma(w)\le C+ \delta_\Gamma(\prod_i g_i w_i g_i^{-1})\le
C+\sum_i \delta_\Gamma(w_i).$$ If $\tau$ is a template, we can break
$w_\tau$ into the $\omega$-triangles and $\omega$-bigons corresponding
to faces and edges of $\tau$ at cost $0$, as in
Figure~\ref{fig:template}.  If $\tau$ is a template for $w$, then
$w_\tau$ can be transformed to $w$ at cost $O(n)$, implying the
following lemma.
\begin{lemma}\label{lem:template}
  Let $w= w_1\dots w_{n}$ be a word of length $n$ and let $\tau$ be a
  template for $w$.  If the $i$th face of $\tau$ has vertices
  $g_{i1},g_{i2},g_{i3}$ and the $j$th edge of $\tau$ has vertices
  $h_{j1},h_{j2}$, then
  $$\delta_G(w)\le \sum_{i}
  \delta_G(\Delta_\omega(g_{i1},g_{i2},g_{i3}))+\sum_j \delta_G(\omega(h_{j1}^{-1}h_{j2})\omega(h_{j2}^{-1}h_{j1}))+O(n).$$
\end{lemma}

Many Dehn function bounds involve a divide-and-conquer strategy which
breaks a complicated word into smaller, simpler words, and templates
are useful to describe such strategies.  For example, one
divide-and-conquer strategy uses the template in
Figure~\ref{fig:dyadic} to build a filling of arbitrary words in a
group out of $\omega$-triangles.  A strategy like this is used, for
instance, in \cite[5.A$''_3$]{GroAII}, \cite{LeuzPitQuad}, and
\cite{deCorTess}; in fact, the following lemma is essentially
equivalent to Lemma~4.3 in \cite{deCorTess}.
\begin{lemma}\label{lem:dyadic}
  If there is an $\alpha>1$ such that for all $h_i\in H$ such that
  $$\trdf{H,\omega}(n)\lesssim n^\alpha,$$
  then
  $$\reldehn{H}{G}(n)\lesssim n^\alpha.$$
\end{lemma}
\begin{figure}
  \includegraphics[width=2.5in]{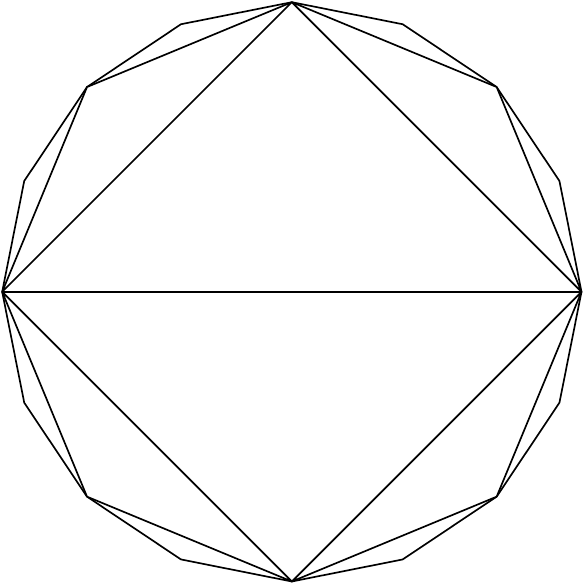}
  \caption{\label{fig:dyadic}A dyadic template}
\end{figure}
\begin{proof}
  Let $S_0\subset S$ be a generating set for $H$, as in the definition
  of $\reldehn{H}{G}$.  Without loss of generality, we may assume that
  the identity $I$ is in $S_0$.  Let $w=w_1\dots w_n$.  It suffices to
  consider the case that $n=2^k$ for some $k\in \Z$; otherwise, we may
  pad $w$ with the letter $I$ until its length is a power of 2.  Let
  $w(i)=w_1\dots w_i$.  Let $\tau$ be the template consisting of
  $2^k-2$ triangles as in Fig.~\ref{fig:dyadic}, where the
  vertices of $\tau$ are labeled by $w(i)$.

  Each triangle of $\tau$ has vertices labeled
  $$w(i2^j),w((i+1/2)2^{j}),w((i+1)2^j)$$
  for some $1\le j<k$ and $0\le i < 2^{-j}n$, which are separated by
  distances at most $2^{j}$.  By the hypothesis, the corresponding
  $\omega$-triangle has a filling of area $O(2^{\alpha j})$.
  Similarly, each edge has vertices labeled $w(i2^j)$ and
  $w((i+1)2^j)$, and corresponds to an $\omega$-bigon which can be
  filled at cost $O(2^{\alpha} j)$.  There are $\sim 2^{-j}n$ bigons
  and edges of size $2^j$, so after summing all the contributions, we
  find that $\delta_H(w)\lesssim n^\alpha$.
\end{proof}

\section{Sketch of proof}\label{sec:sketchProof}

Note that since $\SL(p;\Z)$ is not hyperbolic when $p\ge 3$, its Dehn
function is at least quadratic.  To prove Theorem~\ref{thm:mainThm},
it suffices to show that any word in $\SL(p;\Z)$ has a quadratic
filling.  We proceed by induction on subgroups of $\SL(p;\Z)$.  Very
roughly, we decompose words in $\SL(p;\Z)$ into words in subgroups of
$\SL(p;\Z)$ and then repeat the process inductively to get a filling
of the original word.  We reduce in two main ways.  First, a word in
$\SL(p;\Z)$ corresponds to a curve in the symmetric space
$\cE=\SL(p;\R)/\SO(n)$, and since $\cE$ is nonpositively curved, it
has a filling of quadratic area.  By breaking this filling into pieces
lying in different horoballs, we can break the original word into
pieces lying in maximal parabolic subgroups.

Second, a parabolic subgroup of $\SL(p;\Z)$ is conjugate to an upper
triangular subgroup, and can be written as a semidirect product of a
unipotent group (the off-diagonal part) and a product of
$\SL(q;\Z)$'s (the diagonal blocks).  We use techniques like those
used by Leuzinger and Pittet \cite{LeuzPitQuad} to reduce words in
$\SL(p;\Z)$ to words in the diagonal blocks.  Since each diagonal
block is smaller than the original matrix, repeating these two steps
eventually simplifies the word.

We describe this process more rigorously in the following lemmas.  In
all of these lemmas, $\omega$ will represent a normal form for
$\SL(p;\Z)$; we will define $\omega$ in Section~\ref{sec:normalForms}.
One key property of $\omega$ will be that it is a product of words
representing elementary matrices, which we call \emph{shortcuts}.
These shortcuts are based on the constructions in \cite{LMRComptes}.
Lubotzky, Mozes, and Raghunathan showed that the transvection
$e_{ij}(x)$ can be represented by a word of length logarithmic in $x$;
we denote this word by $\short{e}_{ij}(x)$ and call it a
\emph{shortcut for $e_{ij}(x)$.}  If $H\subset \SL(p;\Z)$, we say that
$w$ is a \emph{shortcut word} in $H$ if we can write $w=\prod_{i=1}^n
w_i$, where each $w_i$ is either a diagonal matrix in $H$ or a
shortcut $\short{e}_{a_ib_i}(x_i)$ where $e_{a_ib_i}(x_i)\in H$.  Our
normal form $\omega$ will express elements $g\in \SL(q;\Z)$ as
shortcut words, and if $g\in \SL(q;\Z)$ or $g\in U(s_1,\dots,s_k)$,
then $\omega(g)$ will be a shortcut word in $\SL(q;\Z)$ or in
$U(s_1,\dots,s_k)$ respectively.

First, we will break loops in $\SL(q;\Z)\subset \SL(p;\Z)$ into
$\omega$-triangles with vertices in maximal parabolic subgroups:
\begin{lemma}[Reduction to maximal parabolics]\label{lem:redPara}
  Let $p\ge 5$ and $2<q\le p$.  There is a $c>0$ such that if $w$ is a word in
  $\SL(q;\Z)$ of length $\ell$, then there are words $w_1,\dots, w_k$ such
  that we can break $w$ into the $w_1,\dots, w_k$ at cost $O(\ell)$;
  each $w_i$ either has length $\le c$ or is an $\omega$-triangle with vertices
  in some $U(q_i,q-q_i)$; and
  $$\sum_i \ell(w_i)^2=O(\ell^2).$$

  As a consequence, if
  $$\trdf{U(s,q-s),\omega}(n)\lesssim n^2$$
  for $s=1,\dots, q-1$, then 
  $$\reldehn{\SL(q;\Z)}{\SL(p;\Z)}(n)\lesssim n^2.$$
\end{lemma}

By our choice of $\omega$, each $w_i$ above is a shortcut word in some
parabolic subgroup.  Each parabolic subgroup is a semi-direct product
of a unipotent subgroup and a (virtual) product of copies of
$\SL(q_i;\Z)$, so we fill the triangles obtained in the previous lemma
by reducing them to shortcut words in the diagonal blocks and shortcut
words in the unipotent subgroup.  The word in the unipotent subgroup
can be filled by combinatorial methods, leaving just the words in the
diagonal blocks.

\begin{rem}
  Ideally, we would be able to construct a projection from an
  $\omega$-triangle in a parabolic subgroup $P$ to shortcut words in
  each diagonal block, and thus break an $\omega$-triangle $w$ in $P$
  into one shortcut word for each diagonal block of $P$ at cost
  $O(\ell(w)^2)$.  When $P\ne U(p-1,1)$, this is possible, but when
  $P=U(p-1,1)$, a different method of proof is necessary.
\end{rem}

\begin{lemma}[Reduction to diagonal blocks]\label{lem:redDiag}
  Let $p\ge 5$ and $q<p$.  Let $1\le s_1,\dots, s_k\le q$ be such that
  $\sum_i s_i\le p$ and suppose that $w$ is an $\omega$-triangle with
  vertices in $U(s_1,\dots,s_k)$ of length $\ell$.  There are words
  $w_1,\dots, w_n$ such that we can break $w$ into the $w_i$'s at cost
  $O(\ell^2)$.  Furthermore, for all $i$ there is a $q_i<q$ such that
  $w_i$ is a shortcut word in $\SL(q_i;\Z)$, and
  $$\sum_i \ell(w_i)^2=O(\ell^2).$$
\end{lemma}

To apply Lemma~\ref{lem:redPara} to these $w_i$ and complete the
induction, we need to replace these shortcut words with words in
$\SL(q;\Z)$.  When $q$ is sufficiently large, this can be done at
quadratic cost.
\begin{lemma}[Moving shortcuts into subgroups]\label{lem:relToTri}
  Let $p\ge 5$ and $2<q\le p$.  If $w$ is a shortcut word in
  $\SL(q;\Z)$, there is a word $w'$ in $\SL(q;\Z)$ such that 
  $\ell(w')=O(\ell(w)$ and $\delta_\Gamma(w,w')=O(\ell(w)^2)$.
\end{lemma}

Ultimately, the previous three lemmas break loops in $\SL(p;\Z)$ into
shortcut words in $\SL(2;\Z)$.  Even though $\SL(2;\Z)$ is virtually
free and has linear Dehn function, shortcut words may leave
$\SL(2;\Z)$ and may have quadratic fillings:
\begin{lemma}[Base case]\label{lem:baseCase}
  Let $p\ge 5$ and let $w$ be a shortcut word in $\SL(2;\Z)$ of length
  $\ell$.  Then 
  $$\delta_\Gamma(w)=O(\ell^2).$$
\end{lemma}

These four lemmas prove Theorem~\ref{thm:mainThm}:
\begin{proof}[Proof of Theorem~\ref{thm:mainThm}]
  We claim that if $w$ is a shortcut word in $\SL(q;\Z)$, then 
  $$\delta_\Gamma(w)=O(\ell^2).$$
  This implies that 
  $$\reldehn{\SL(q;\Z)}{\SL(p;\Z)}(n)\lesssim n^2,$$
  and when $q=p$, this proves the theorem.

  We proceed by induction.  When $q=2$, the statement is
  Lemma~\ref{lem:baseCase}.  Otherwise, since $q>2$, we can apply
  Lemma~\ref{lem:relToTri} to replace $w$ by a word $w'$ in
  $\SL(q;\Z)$, and apply Lemma~\ref{lem:redPara} and
  Lemma~\ref{lem:redDiag} to break $w'$ into words $w_{i}$, each a
  shortcut word in some $\SL(q_i;\Z)$'s, 
  such that $\sum_i \ell(w_i)^2=O(\ell^2).$  This has cost
  $O(\ell^2)$, and by the inductive hypothesis, the total filling area
  of the $w_i$'s is also $O(\ell^2)$, as desired.
\end{proof}

In the next two subsections, we will describe some of the ideas behind
the proofs of these lemmas.  Then, Lemma~\ref{lem:redPara} will be
proved in Section~\ref{sec:redParaProof}, Lemma~\ref{lem:relToTri}
will be proved in Section~\ref{sec:relToTriProof},
Lemma~\ref{lem:redDiag} will be proved in
Section~\ref{sec:redDiagProof}, and Lemma~\ref{lem:baseCase} will be
proved in Section~\ref{sec:baseCaseProof}.

\subsection{Constructing templates from Lipschitz fillings}
\label{subsec:redParaSketch}
One idea behind the proof of Theorem~\ref{thm:mainThm} is that we can use a
Lipschitz filling of a curve in a symmetric space to construct a
template for a filling of $w$.  In this section, we will sketch how to
use the pattern of intersections between the filling and the horoballs
in the symmetric space to break $w$ into $\omega$-triangles lying in
parabolic subgroups.

If $w$ is a word in $\SL(q;\Z)$, it corresponds to a curve $\gamma_w$
in the non-positively curved symmetric space $\cE=\SL(q;\R)/\SO(q)$ of
length $\ell$, and this curve has a quadratic filling.  Indeed, if
$D^2(\ell)$ is the disc $[0,\ell]\times [0,\ell]$, there is a filling
$f:D^2(\ell)\to \cE$ which has Lipschitz constant at most $2$.  We can
construct $f$ by choosing a basepoint on the curve and contracting the
curve to the basepoint along geodesics.  Choose a Siegel set
$\cS\subset \cE$; this is a fundamental set for the action of
$\SL(q;\Z)$ on $\cE$ (see Sec.~\ref{sec:depthFunction}).  Each point of $\cE$ lies in some
translate of $\cS$; we can define a map $\rho:\cE\to \SL(q;\Z)$ by
sending each point $x$ to a group element $\rho(x)$ such that $x\in
\rho(x)\cS$.  Then, if $\tau$ is a triangulation of $D^2(\ell)$, we
can label each vertex $v$ by the element $\rho(f(v))$.  This is a
template, and if the boundary edges of $\tau$ each have length bounded
by a constant, then the boundary word $w_\tau$ of the template is
uniformly close to $w$.

As a simple application, we will show that for any $q$, the Dehn
function of $\SL(q;\Z)$ is bounded by an exponential function.  It is
straightforward to show that the injectivity radius of $z\in
\cE/\SL(q;\Z)$ shrinks exponentially quickly as $z\to \infty$; that
is, that there is a $c$ such that if $x,y\in \cE$, $d_\cE(I,x)\le r$,
and $d_{\cE}(x,y)\le e^{-c r}$, then
$d_{\SL(q;\Z)}(\rho(x),\rho(y))\le c$.  Let $\tau$ be a triangulation
of $D^2(\ell)$ by triangles with side lengths at most $e^{-2c\ell}$.
If an edge of $\tau$ connects vertices $u$ and $v$, then
$d_{\SL(q;\Z)}(\rho(f(u)),\rho(f(v)))\le c$, so
$$\delta(w_\tau)\le F \delta(3c)+E \delta(2c),$$
where $F$ is the number of faces of $\tau$ and $E$ is the number of
edges.  Since we can construct $\tau$ to have
at most exponentially many triangles, $\delta(w_\tau)\lesssim e^\ell$.

Triangulations with larger simplices lead to larger $\omega$-triangles
but potentially stronger bounds on the Dehn function; for example, in
\cite{YoungQuart}, we used a triangulation by triangles of diameter
$\sim 1$ to prove a quartic bound on $\SL(q;\Z)$ when $q\ge 5$.  The
basic idea behind that proof was that if $x, y\in \cE$ are
sufficiently close together, then either $\rho(x)^{-1}\rho(y)$ is
bounded or it lies in a parabolic subgroup of $\SL(q;\Z)$, so the
methods above produce a template whose triangles all either have
bounded size or lie in a parabolic subgroup.  Furthermore, since each
edge is short, the group elements corresponding to edges satisfy
bounds which make the triangles easy to fill.

We use a similar idea to prove Lemma~\ref{lem:redPara}.  One can show
(see Cor.~\ref{cor:parabolicNbhds}) that if $x$ is deep in the cusp of
$\cM$, i.e., if $r(x)=d_\cE(x,[\SL(p,\Z)]_\cE)$ is large, then there
is a ball around $x$ of radius $\sim r(x)$ which is contained in a
horoball corresponding to a maximal parabolic subgroup of $\SL(q;\Z)$.
In particular, if $d(x,y)\ll r(x)$, then $\rho(x)^{-1}\rho(y)$ lies in
a maximal parabolic subgroup of $\SL(q;\Z)$.

If $f:D^2(\ell)\to \cE$ is a Lipschitz filling of a curve $\gamma$, we
can construct a triangulation of $D^2(\ell)$ where the size of each
triangle is proportional to its distance from the thick part.  By
labeling the vertices of this triangulation as above, we get a
template made of triangles which are either ``small'' or ``large''.
Small triangles are those whose image under $f$ is in the thick part;
since the injectivity radius of $\cE$ is bounded away from zero in the
thick part, the
vertex labels of a small triangle are a bounded distance apart in
$\SL(q;\Z)$.  Large triangles are those whose image is in the thin
part.  The image of a large triangle under $f$ lies in a horoball, and
its vertex labels lie in a conjugate of one of the maximal parabolic
subgroups.  Ultimately, this lets us break words in $\SL(q;\Z)$ into
$\omega$-triangles with vertices in parabolic subgroups.  This is a key step in the proofs of
Lemmas~\ref{lem:redPara} and \ref{lem:baseCase}, and in a special case
of Lemma~\ref{lem:redDiag}.

\subsection{Shortcuts in $\SL(p;\Z)$}

Another idea behind the proof of Theorem~\ref{thm:mainThm} is the idea of
\emph{shortcuts}, words of length $\sim \log n$ which represent
transvections in $\SL(p;\Z)$ with coefficients of order $n$.  These
shortcuts are a key ingredient in the construction of the normal form
$\omega$.  Transvections satisfy Steinberg relations, and one of
the key combinatorial lemmas (Lemma~\ref{lem:infPres}) states that when these Steinberg
relations are written in terms of shortcuts, the resulting words have
quadratic fillings.

Our shortcuts are based on constructions of Lubotzsky, Mozes, and
Raghunathan \cite{LMRComptes}, who used them to show that distances in
the word metric on $\SL(p;\Z)$ are comparable to distances in the
riemannian metric on the symmetric space $\SL(p;\Z)/\SO(p)$ when $p\ge
3$.  In particular, if $M\in \SL(p;\Z)$ is a matrix with coefficients
bounded by $\|M\|_\infty$, there is a word $w$ which represents $M$ as
a product of $\sim \log \|M\|_\infty$ generators of $\SL(p;\Z)$.  They
construct this $w$ by decomposing $M$ into a product of transvections
with integer coefficients, then writing each transvection as a word in
$\SL(p;\Z)$.  This can be done efficiently because unipotent subgroups
of $\SL(p;\Z)$ are exponentially distorted; a transvection with
$L^\infty$ norm $N$ can be written as a word of length $\sim \log N$.
In fact, a transvection can be written as a word of length $\sim \log
N$ in many ways.  

One advantage of working with $\SL(p;\Z)$ instead of an arbitrary
lattice in a high-rank Lie group is that shortcuts in $\SL(p;\Z)$ can
be written with just a few generators, and that many of the generators
of $\SL(p;\Z)$ commute.  For example, if we define $e_{ij}(x)$ to be
the elementary matrix obtained by replacing the $(i,j)$-entry of the
identity matrix by $x$, there is a word $\short{e}_{13}(x)$ in the
alphabet $\{e_{12},e_{21},e_{13},e_{23}\}$ which represents
$e_{13}(x)$ and has length $\sim \log |x|$.  If $w$ is a product of
generators that commute with this alphabet, it's easy to fill words
like $[\short{e}_{13}(x), w]$.  Furthermore, when $p\ge 5$, different
ways of constructing shortcuts are close together; we can arrange
things so that if $\short{e}$ and $\short{e}'$ are shortcuts for the
same elementary matrix written in different alphabets, then
\begin{equation}\label{eq:differentShorts}
  \delta(\short{e}, \short{e}')\lesssim \ell(\short{e})^2.
\end{equation}
This lets us write elementary matrices in terms of whichever alphabet
is most convenient.  The fact that \eqref{eq:differentShorts} is not
true when $p=4$ is the biggest obstacle to extending these techniques
to $\SL(4;\Z)$; an analogue of \eqref{eq:differentShorts} for
$\SL(4;\Z)$ would lead to a polynomial bound on its Dehn function.

\noindent \textit{Remark on notation:} We will generally use hats
to denote shortcuts, so $e_{ij}(x)$ and $u(V)$ will denote
unipotent matrices and $\short{e}_{ij}(x)$ and $\short{u}(V)$
will denote words of logarithmic length that represent the
corresponding matrices.  

We prove Lemma~\ref{lem:redDiag} using these shortcuts.  The
normal form $\omega$ expresses elements of $\SL(p;\Z)$ in terms of
shortcuts, and the proof of Lemma~\ref{lem:redDiag} mostly consists of
combinatorial calculations involving these shortcuts.  For example,
as mentioned above, one step in the proof involves constructing
fillings of Steinberg 
relations.  The elementary matrices $e_{ij}(x)$ satisfy relations
like $[e_{ij}(x),e_{kl}(y)]=I$ and $[e_{ij}(x),e_{jk}(y)]=e_{ik}(xy)$,
so the corresponding products of shortcuts (e.g.,
$[\short{e}_{ij}(x),\short{e}_{kl}(y)]$) are words representing the
identity.  By rewriting these shortcuts in appropriate alphabets, we
can fill these words efficiently.
%weird placement

\section{Siegel sets and the depth function}\label{sec:depthFunction}

Let $G=\SL(p;\R)$ and $\Gamma=\SL(p;\Z)$.  Given a fundamental set $F$
for the action of $\Gamma$ on $\cE$, one can construct a map $\cE\to
\Gamma$ which sends each point $x$ of $\cE$ to an element $g\in
\Gamma$ such that $x\in gF$.  In general, this map need not be
well-behaved, but if $F$ is a Siegel set, this map has many useful
properties.  In this section, we will define a Siegel set $\cS$ and
describe some of its properties.  Note that the constructions in this
section generalize to many reductive and semisimple Lie groups with
the use of precise reduction theory, but we will only state the
results for $\SL(p;\Z)$, as stating the theorems in full generality
requires a lot of additional background (see Sec.~\ref{sec:open} for
some discussion of the general case).

Let $\diagmat(t_1,\dots, t_p)$ be the diagonal matrix with entries
$(t_1,\dots, t_p)$.  Let $A$ be the set of diagonal matrices in $G$
and if $\epsilon>0$, let
$$A^+_{\epsilon}=\{\diagmat(t_1,\dots, t_p)\mid \prod t_i=1, t_i > 0,
t_i\ge \epsilon t_{i+1}\}.$$ Let $\cM=\SL(p;\Z)\backslash \cE$.  One
of the main features of $\cM$ is that it is Hausdorff equivalent to
$A^+_{\epsilon}$; our main goal in this section is to describe this
Hausdorff equivalence and its ``fibers''.  Let $N$ be the set of upper
triangular matrices with 1's on the diagonal and let $N^+$ be the
subset of $N$ with off-diagonal entries in the interval $[-1/2,1/2]$.
Translates of the set $N^+A^+_\epsilon$ are known as Siegel sets.  The
following properties of Siegel sets are well known (see for instance
\cite{BorHar-Cha}).
\begin{lemma}\label{lem:redThe}\ \\
  There is an $1>\epsilon_{\cS}>0$ such that if we let 
  $$\cS:=[N^+A^+_{\epsilon_{\cS}}]_\cE\subset \cE,$$
  then
  \begin{itemize}
  \item $\Gamma\cS=\cE$. \label{lem:redThe:cover}
  \item There are only finitely many elements $\gamma\in \Gamma$
    such that $\gamma \cS \cap \cS \ne \emptyset$.  \label{lem:redThe:fundSet}
  \end{itemize}
\end{lemma}
In particular, the quotient map $\cS\to \cM$ is a surjection.  We
define $A^+:=A^+_{\epsilon_{\cS}}$.

The inclusion $A^+ \hookrightarrow \cS$ is a Hausdorff equivalence.
That is, if we give $A$ the riemannian metric inherited from its
inclusion in $G$, so that
$$d_{A}(\diagmat(d_1,\dots,d_p),\diagmat(d'_1,\dots,d'_p))=\sqrt{\sum_{i=1}^p
  \left|\log \frac{d'_i}{d_i}\right|^2},$$
then
\begin{lemma}[{\cite{JiMacPherson}}] \label{lem:easyHausdorff} There
  is a $c$ such that if $n\in N^+$ and $a\in A^+$, then
  $d_\cE([na]_\cE,[a]_\cE)\le c$.  In particular, if $x\in \cS$, then
  $d_\cE(x,[A^+]_\cE)\le c$.

  Furthermore, if $x,y\in A^+$, then $d_{A}(x,y)=d_{\cS}(x,y)$.
\end{lemma}
\begin{proof}
  For the first claim, note that if $x=[na]_\cE$, then $x=[a(a^{-1} n
  a)]_\cE$, and $a^{-1}na\in N$.  Furthermore,
  $$\|a^{-1}na\|_\infty\le \epsilon_\cS^{-p},$$
  so 
  $$d_\cE([x]_\cE,[a]_\cE)\le d_G(I,a^{-1}na)$$
  is bounded independently of $x$.

  For the second claim, we clearly have $d_{A}(x,y)\ge
  d_{\cS}(x,y)$.  For the reverse inequality, it suffices to note that
  the map $\cS\to A^+$ given by $na\mapsto a$ for all $n\in N^+$,
  $a\in A^+$ is distance-decreasing.
\end{proof}
Siegel conjectured that the quotient map from $\cS$ to $\cM$ is also a
Hausdorff equivalence, that is:
\begin{thm}\label{thm:SiegConj}
  There is a $c'$ such that if $x,y\in \cS$, then 
  $$d_{\cE}(x,y)-c'\le d_{\cM}([x]_\cM,[y]_\cM)\le d_{\cE}(x,y)$$
\end{thm}
Proofs of this conjecture can be found in \cite{Leuzinger,Ji,Ding}.
As a consequence, the natural quotient map $A^+\to \cM$ is a Hausdorff
equivalence.

Since $\cS$ is a fundamental set, any point $x\in \cE$ can be written
(possibly non-uniquely) as $x=[\gamma na]_\cE$ for some $\gamma\in
\Gamma$, $n\in N^+$ and $a\in A^+$.  Theorem~\ref{thm:SiegConj} implies that these
different decompositions are a bounded distance apart:
\begin{cor}[{see \cite{JiMacPherson}, Lemmas 5.13, 5.14}]\label{cor:Hausdorff}
  There is a constant $c''$ such that if $x,y\in \cM$,
  $n,n' \in N^+$ and $a,a'\in A^+$ are
  such that $x=[na]_\cM$ and $y=[n'a']_\cM$, then
  $$|d_{\cM}(x,y)- d_{A}(a,a')|\le c''.$$ 

  In particular, if $[\gamma na]_\cE=[\gamma' n' a']_\cE$ for some
  $\gamma,\gamma'\in \Gamma$, then $d_{A}(a,a')\le c''.$
\end{cor}
\begin{proof}
  Note that by Lemma~\ref{lem:easyHausdorff}, 
  $$d_{\cM}(x,[a]_\cM)\le d_{\cE}([na]_\cE,[a]_\cE)\le c$$
  and likewise $d_{\cM}(y,[a']_\cM)\le c$.  Furthermore, by the
  theorem and the lemma, 
  $$d_A(a,a')-c'=d_\cS(a,a')-c' \le d_{\cM}([a]_\cM,[a']_\cM) \le
  d_A(a,a'),$$
  so if we let $c''=c'+2c$, the corollary follows.
\end{proof}

Let $\rho:\cE\to \Gamma$ be a map such that $\rho(\cS)=I$ and $x\in
\rho(x)\cS$ for all $x$.  Any point $x\in \cE$ can be uniquely written
as $x=[\rho(x)na]_\cE$ for some $n\in N^+$ and $a\in A^+$.  Let
$\phi:\cE\to A^+$ be the map $[\rho(x)na]_\cE\mapsto a$.  There are
many choices for $\rho$, but, by Cor.~\ref{cor:Hausdorff}, they only
affect the definition of $\phi$ by a bounded amount.  If
$\phi(x)=\diagmat(a_1,\dots,a_p)$, let $\phi_i(x)=\log a_i$.  If
$x,y\in \cE$, then $|\phi_i(x)-\phi_i(y)|\le d_\cE(x,y)+ c''$; let
$c_\phi:=c''$.\label{sec:cphi}

Define the \emph{depth function}
$r:\cE\to \R^+$, $r(x)=d_{\cM}([x]_\cM,[I]_{\cM})$ which measures the
distance between $x$ and the thick part of $\cE$; the results above
imply that 
$$r(x)\sim \log \|\phi(x)\|_2\sim \phi_1(x)-\phi_p(x).$$
Since the injectivity radius of the cusp decreases exponentially as
one gets further away from $\Gamma$, the distortion of $\rho$ depends
on depth:
\begin{lemma} \label{lem:depthLength}
  There is a $c$ such that if $x,y\in \cE$, then 
  $$d_\Gamma(\rho(x),\rho(y))\le c(d_\cE(x,y)+r(x)+r(y))+c$$
\end{lemma}
\begin{proof}
  By Thm.~\ref{thm:SiegConj}, there is a $c_0$ such that
  $d_\cE([\rho(x)]_\cE,x)\le r(x)+c_0,$
  so 
  $$d_\cE([\rho(x)]_\cE,[\rho(y)]_\cE) \le r(x)+r(y)+d_\cE(x,y)+2c_0.$$
  The lemma follows by Thm.~\ref{thm:LMR}.
\end{proof}

The depth function governs $\rho$ in other ways as well.  Recall that
if $x\in \cE$ and $\tilde{x}\in G$ is a representative of $x$, we can
construct a lattice $\Z^p \tilde{x}\subset \R^p$ and a different choice
of $\tilde{x}$ corresponds to a lattice that differs by a rotation.
When $r(x)$ is large, then the lattice has short vectors.  If $y$ is
close to $x$, then vectors that are short in $\Z^p \tilde{y}$ are also
short in $\Z^p
\tilde{x}$.  These vectors define a subspace in $\Z^p$, and
$\rho(x)^{-1}\rho(y)$ must preserve that subspace; i.e.,
$\rho(x)^{-1}\rho(y)$ must lie in a parabolic subgroup.  The next
lemmas make this argument formal.  
If $x\in \cE$ and $\tilde{x}\in G$ is such that $x=[\tilde{x}]_\cE$, 
let
$$V(x,r)=\langle v\in \Z^p\mid \|v \tilde{x}\|_2 \le r  \rangle;$$
we call this the \emph{$r$-short subspace of $x$} and it is
independent of the choice of $\tilde{x}$.  Let $z_1,\dots, z_p\in
\Z^p$ be the standard generators of $\Z^p$.

\begin{lemma}\label{lem:phiFlag}
  There is a $c_V>0$ depending only on $p$ such that if $x=[\gamma n
  a]_\cE$, where $\gamma\in \Gamma$, $n\in N^+$, 
  $$a=\diagmat(a_1,\dots,a_p)\in A^+,$$
  and
  $$e^{c_V}a_{k+1}<r<e^{-c_V}a_{k},$$
  then $V(x,r)=Z_k\gamma^{-1}$, where $Z_k:=\langle z_{k+1},\dots, z_p
  \rangle$.
\end{lemma}
\begin{proof}
  Note that $V(\gamma' x',r)=V(x',r){\gamma'}^{-1}$, so we may assume that
  $\gamma=I$ without loss of generality.  Let $n=\{n_{ij}\}\in N^+$
  and let $\tilde{x}=na$.  We have
  \begin{align*}
    z_j \tilde{x} &= z_j na \\
    &= a_j z_j+\sum_{i=j+1}^p n_{ji} z_ia_{i}.
  \end{align*}
  Since $a_{i+1}\le a_i \epsilon_{\cS}^{-1}$, we have
  $a_i\le a_{k+1}\epsilon_{\cS}^{-p}$ for $i\ge k+1$ and 
  $a_i\ge a_{k}\epsilon_{\cS}^{p}$ for $i\le k$.  Since
  $|n_{ji}|\le 1/2$ when $i>j$,
  we have
  $$\|z_j \tilde{x}\|_2 \le a_{k+1}\sqrt{p}\epsilon_{\cS}^{-p}$$
  when $j>k$.
  Thus 
  $$V(x,a_{k+1}\sqrt{p}\epsilon_{\cS}^{-p})\supset Z_k.$$  

  On the
  other hand, assume that $v\not \in Z_k$, and let $v=\sum_i v_i z_i$ for some
  $v_i\in \Z$.  Let $j$ be the smallest integer such that $v_j\ne 0$; by
  assumption, $j\le k$.  The $z_{j}$-coordinate of $v\tilde{x}$ is
  $v_{j} a_{j}$, so
  $$\|v \tilde{x} \|_2 \ge a_{j} > a_{k} \epsilon_{\cS}^{p}$$
  and thus if $t<a_{k} \epsilon_{\cS}^{p}$, then $V(x,t)\subset Z_k$.
  Therefore, if 
  $$a_{k+1}\sqrt{p}\epsilon_{\cS}^{-p}\le t<a_{k} \epsilon_{\cS}^{p},$$
  then $V(\tilde{x},t)=Z_j$.  We can choose $c_V=\log \sqrt{p}\epsilon_{\cS}^{-p}$.
\end{proof}
In particular, since different choices in the construction of
$\rho(x)$ must still lead to the same $V(x,r)$, this means that if
$\phi_k(x)-\phi_{k+1}(x)$ is sufficiently large, then 
different choices of $\rho(x)$ must differ by an element of $U(k,p-k)$.
The next lemma extends this by noting that nearby points of $\cE$ must
have the same $r$-short subspaces:
\begin{lem} \label{lem:parabolicSubsets}
  Let 
  $$B_j(c):=\{x\in \cE\mid \phi_j(x)-\phi_{j+1}(x)>c\}.$$
  There is a $c>0$ depending only on $p$ such that if $1\le j<p$,
  $x,y\in \cE$ are in the same connected component of $B_j(c)$, and
  $g,h\in \Gamma$ are such that $x\in g\cS$, $y\in h\cS$, then
  $g^{-1}h\in U(j,p-j)$.  In particular, $\rho(x)^{-1}\rho(y)\in
  U(j,p-j)$.
\end{lem}
\begin{proof}
  Define $s(z)=\exp\frac{\phi_{j+1}(z)+\phi_j(z)}{2}$, so that if
  $z\in B_j(c)$, then
  $$e^{c/2}e^{\phi_{j+1}(z)}<s(z)<e^{-c/2}e^{\phi_{j}(z)}.$$
  We will show that if $c$ is sufficiently large, then the function
  $z\mapsto V(z,s(z))$ is constant on each connected component of
  $B_j(c)$.  Let $c=2(c_V+ c_\phi+1)$.

  Note that if $z,z'\in \cE$, if $\tilde{z},\tilde{z}'\in G$ are
  representatives of $z$ and $z'$, and if $v\in \Z^p$, then
  $$\left|\log \|v \tilde{z}\|_2-\log \|v \tilde{z}'\|_2\right| \le  d_\cE(z,z').$$
  Furthermore, if $z,z'\in B_j(c)$, then 
  $$|\log s(z)-\log s(z')|\le d_\cE(z,z')+c_\phi.$$

  Fix $z$.  Since $z=[\rho(z)n\phi(z)]_\cE$ for some $n\in
  N^+$, Lemma~\ref{lem:phiFlag} states that if
  $$\exp(c_V+\phi_{j+1}(z))<r<\exp(-c_V+\phi_j(z)),$$
  then $V(z,r)=Z_j\rho(z)^{-1}.$  In particular, if
  $$s(z)e^{-c_\phi-1}<r<s(z)e^{c_\phi+1}$$
  then $V(z,r)=Z_j\rho(z)^{-1}.$
  So if $z'\in\cE$ is distance at most $1/2$ from $z$, then
  $$|\log s(z)-\log s(z')|\le 1/2+c_\phi,$$
  and 
  $$V(z,s(z) e^{-c_\phi-1})\subset V(z,s(z') e^{-1/2})\subset V(z',s(z'))\subset V(z,s(z) e^{c_\phi1}),$$
  so $V(z',s(z'))=V(z,s(z))$.  Thus the function
  $z\mapsto V(z,s(z))$ is locally constant at each point of $B_j(c)$
  and thus it is constant on each connected component of $B_j(c)$.

  Say that $x\in B_j(c)$ and that $x\in g\cS$ for some $g\in \Gamma$.  We
  can write $x=[gna]_\cE$ for some $n\in N^+$, $a=\diagmat(a_1,\dots,a_p)\in A^+$, and
  Cor.~\ref{cor:Hausdorff} implies that $d_{A}(a,\phi(x))\le c_\phi$
  and thus $|\log a_i-\phi_i(x)|\le c_\phi$ for all $i$.  In particular, 
  $$e^{c_V}a_{j+1}<s(x)<e^{-c_V}a_{j},$$
  so Lemma~\ref{lem:phiFlag} shows that $V(x,s(x))=Z_jg^{-1}.$

  In particular, if $y\in B_j(c)$ is in the same connected component
  as $x$ and if $y\in h\cS$, then $V(x,s(x))=V(y,s(y))$, so $Z_jg^{-1}=Z_jh^{-1},$
  and $g^{-1}h$ stabilizes $Z_j$.  This implies $g^{-1}h\in U(j,p-j)$ as desired.
\end{proof}

Since $r(x)\sim \phi_1(x)-\phi_p(x),$ if $r(x)$ is large, then $x\in
B_j(c)$ for some $j$.  As a consequence, if $x$ is deep in the cusp of
$\cM$, there is a large ball $B$ around $x$ such that $\rho(B)$ is
contained in a coset of a maximal parabolic subgroup (see Fig.\
\ref{fig:cusps}).  

\begin{figure}
  \def\svgwidth{4in}
  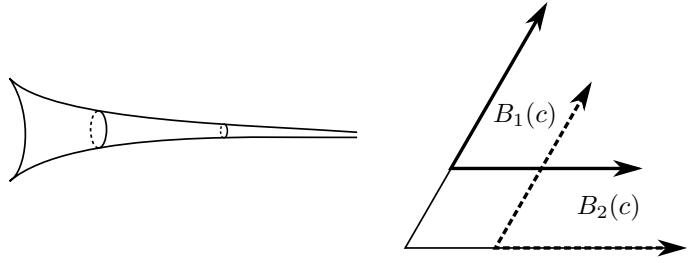
  \caption{\label{fig:cusps} Left:  $\cM$ for $p=2$.  Right: $A^+$ for
    $p=3$.  When $p=2$, $A^+$ is 1-dimensional, and the cusp has fundamental
    group $\Z$, conjugate to a parabolic subgroup.  When $p=3$, $A^+$
    is 2-dimensional, and the cusp is more complicated.  The marked
    regions correspond to $B_1(c)$ (bounded by solid lines) and
    $B_2(c)$ (bounded by dashed lines).  The images in $\SL(3;\Z)$ of
    the fundamental groups of $B_1(c)$ and $B_2(c)$ are parabolic subgroups of
    $\SL(3;\Z)$. }
\end{figure}

We claim:
\begin{cor} \label{cor:parabolicNbhds} There is a $c'>0$ such that if
  $x\in \cE$, $r(x)>c'$, and $B\subset \cE$ is the ball of radius
  $\frac{r(x)}{4p^2}$ around $x$, then $\rho(B)\subset g U(j,p-j)$ for
  some $g\in \Gamma$ and $1\le j\le p-1$.  Indeed, if $h\in \Gamma$ is such that $h\cS\cap B\ne
  \emptyset$, then $h\in g U(j,p-j)$.
\end{cor}
\begin{proof}
  We will find a $c'$ such that if $r(x)>c'$, then
  $$\frac{r(x)}{4p^2}<\frac{\phi_j(x)-\phi_{j+1}(x)-2c_\phi-c}{2}$$
  for some $j$, where $c$ is as in Lemma~\ref{lem:parabolicSubsets}.  If
  $y\in B$, then $x$ and $y$ are connected by a geodesic segment of
  length at most $r(x)/{4p^2}$, and if $z$ is a point on that segment,
  then
  $$|(\phi_j(z)-\phi_{j+1}(z))-(\phi_j(x)-\phi_{j+1}(x))|\le 2c_\phi+
  2\frac{r(x)}{4p^2},$$ so $z\in B_j(c)$, and $x$ and $y$ satisfy the
  conditions of Lemma~\ref{lem:parabolicSubsets}.

  Since $\sum_i \phi_i(x)=0$, we have
  $$|\phi_i(x)|\le p \max_j |\phi_j(x)-\phi_{j+1}(x)|$$
  for all $j$.  By Corollary~\ref{cor:Hausdorff}, 
  $$r(x)\le c_\phi+d_{A^+}(I,\phi(x)) \le c_\phi+ p^2 \max_j
  |\phi_j(x)-\phi_{j+1}(x)|,$$
  so there is a $j$ such that
  $$|\phi_j(x)-\phi_{j+1}(x)|\ge \frac{r(x)-c_\phi}{p^{2}}.$$
  However, by the definition of $A^+$,
  $\phi_j(x)-\phi_{j+1}(x)>\log \epsilon_{\cS}$ (see
  Lemma~\ref{lem:redThe}), so if $r(x)$ is sufficiently large, then 
  $$\phi_j(x)-\phi_{j+1}(x)\ge \frac{r(x)-c_\phi}{p^{2}}.$$
  If $r(x)$ is even larger, then
  $$\frac{r(x)}{4p^2}<\frac{\phi_j(x)-\phi_{j+1}(x)-2c_\phi-c}{2}$$
  as desired.
\end{proof}

In the next section, we will use this property of $r(x)$ to construct
a template.

\section{Reducing to maximal parabolic subgroups}\label{sec:redParaProof}
In this section, we will prove Lemma~\ref{lem:redPara} by constructing
a disc in $\cE=\SL(p;\R)/\SO(p)$, a triangulation of that disc, and a
template based on that triangulation.  The basic idea of the proof is
sketched in Sec.~\ref{subsec:redParaSketch}:  any curve in
$\cE$ can be filled by a Lipschitz disc, which might travel through
the thin part of $\cE$.  We triangulate the disc so that each triangle
lies in a single horoball, label the triangulation to get a template,
then bound the lengths of the words in the template.

Let $r:\cM\to \R$ be the depth function defined in
Sec.~\ref{sec:depthFunction}.  We will prove the following:
\begin{lem}\label{lem:templateExist}
  Let $q\in \Z$ and $q\ge 2$.  If $w=w_1\dots w_\ell$ is a word in
  $\Gamma$ which represents the identity, then there is a
  triangulation $\tau$ of a square of side length $\sim \ell$ with
  straight-line edges and a labelling of the vertices of $\tau$ by
  elements of $\Gamma$ such that the resulting template satisfies:
  \begin{enumerate}
  \item \label{lem:templateExist:length} If $g_1,g_2$ are the labels of an edge $e$ in the template,
    then 
    $$d_\Gamma(g_1,g_2)=O(\ell(e)),$$
    where $\ell(e)$ is the length of $e$ as a segment in the square.
  \item \label{lem:templateExist:para} There is a $c>0$ independent of
    $w$ such that if $g_1, g_2, g_3\in \Gamma$
    are the labels of a triangle in the template, then either
    $\diam\{g_1,g_2,g_3\}\le c$ or there is a $1\le k<q$ such that all of
    the $g_i$ are contained in the same coset of $U(k,q-k)$.
  \item \label{lem:templateExist:area}$\tau$ has $O(\ell^2)$
    triangles, and if the $i$th triangle of $\tau$ has vertices
    labeled $(g_{i1},g_{i2},g_{i3})$, then
    $$\sum_{i}(d_\Gamma(g_{i1},g_{i2})+d_\Gamma(g_{i1},g_{i3})+d_\Gamma(g_{i2},g_{i3}))^2=O(\ell^2).$$

    Similarly, if the $i$th edge of $\tau$ has vertices labeled
    $h_{i1}, h_{i2}$, then 
    $$\sum_{i}d_\Gamma(h_{i1},h_{i2})^2=O(\ell^2).$$
  \end{enumerate}
\end{lem}
This immediately implies Lemma~\ref{lem:redPara}.

We construct this template in the way described in
Section~\ref{subsec:redParaSketch}: we start with a filling of $w$ by
a Lipschitz disc $f:D^2\to \cE$, then construct a template for $w$ by
triangulating the disc and labelling its vertices using $\rho$.  We
ensure that properties \ref{lem:templateExist:length} and \ref{lem:templateExist:para} hold by carefully controlling the
lengths of edges.  If edges are too long, then property
\ref{lem:templateExist:para} will not hold.  On the other hand, if
$x,y\in \cE$, then $\rho(x)$ and $\rho(y)$ may be separated by up to
$\sim r(x)+r(y)+d_\cE(x,y)$, so if edges are too short, then
\ref{lem:templateExist:length} will not hold.  For both these
conditions, it suffices to construct a triangulation so that the
triangle containing $x$ has diameter roughly proportional to the depth
function $r(x)$.

We will need the following lemma, which cuts a square of side $2^k$
into dyadic squares whose side lengths are comparable to a Lipschitz
function $h$; that is, there is a $c>0$ such that if $S$ is one of the
subsquares, with side length $\sigma(S)$, then 
$$c^{-1} \min\{2^k,\min_{x\in S}h(x)\}\le \sigma(S)\le c \max_{x\in  S}h(x).$$
This is similar to the
decomposition used to prove the Whitney extension theorem, which,
given a closed set $K$, decomposes $\R^n\setminus K$ into cubes such
that for each cube $S$, the side length $\sigma(S)$ of $S$ satisfies
$\sigma(S)\sim d(S,K)$.

A \emph{dyadic square} is a square of the form
$$S_{i,j,s}:=[i 2^s,(i+1) 2^s]\times [j 2^s,(j+1) 2^s]$$
for some $i,j,s\in \Z$, $s\ge 0$.  We denote the set of dyadic squares
contained in $\sqD(t)$ by $\D_t$.  If $S$ is a square, let $\sigma(S)$
be its side length.
\begin{lemma}\label{lem:adaptiveSquares}
  Let $t=2^k$, $k\ge 0$, let $D^2(t)=[0,t]\times[0,t]$, and let
  $h:\sqD(t)\to \R$ be a $1$-Lipschitz function such that $h(x)\ge 1$
  for all $x$.  There is a set of dyadic squares $U$ such that:
  \begin{enumerate}
  \item $U$ covers $D^2(t)$, and any two squares in $U$ intersect only
    along their edges.  
  \item If $S\in U$, then 
    $$\min\biggl\{\frac{h(x)}{6},\frac{t}{2}\biggr\} \le \sigma(S) \le h(x)$$
    for all $x\in S$.
  \item Each square in $U$ neighbors no more than $16$ other squares.
  \end{enumerate}
\end{lemma}
\begin{proof}
  The dyadic squares can be arranged in a rooted tree whose root is
  $D^2(t)$ so that the children of a dyadic square of side length
  $2^s$, $s>1$ are the four squares of side length $2^{s-1}$ which it
  contains.  If $S$ is a dyadic square, let $a(S)$ be its parent square.  If $S$ and $T$ are
  dyadic squares whose interiors intersect, then one must be the
  ancestor of the other.  That is, either $S\subset T$ and $T=a^k(S)$
  for some $k$ or vice versa.

  Let
  $$U_0:=\{S \mid S\in \D_t \text{ and }  \sigma(S) \le h(x)\text{ for
    all } x\in S\}$$
  and let $U$ be the set of maximal elements in $U_0$:
  $$U:=\{S \mid S\in U_0 \text{ and } a^k(S)\not\in U_0 \text{ for all
    $k$}\}.$$
  We claim that this is the desired cover.

  First, we show that it is a cover of $\sqD(t)$.  If $x\in \sqD(t)$, then $x\in S$
  for some $S\in \D_t$ with $\sigma(S)=1$.  Since $h(z)\ge 1$ for all
  $z\in \sqD(t)$, we know that $S\in U_0$.  If $n$ is the largest
  integer such that $a^n(S)\in U_0$, then $a^n(S)\in U$.  So $x$ is
  contained in a square of $U$, and since $x$ was arbitrary, $U$ is a
  cover of $\sqD(t)$.

  Furthermore, if $S,T\in U$ intersect along more than an
  edge, then one must be an ancestor of the other.  Since $S$ and $T$
  are maximal elements of $U_0$, this means that $S=T$.

  Next, we prove property 2.  By the definition of $U_0$, if $S\in U$,
  then $\sigma(S) \le h(x)$ for all $x\in S$, so it remains to prove
  the lower bound on $\sigma(S)$.  If $S=D^2(t)$, then $\sigma(S)\ge
  t/2$, so the bound holds; otherwise, if $S\in U$, then $a(S)\not\in
  U_0$ by the definition of $U$, so there must be some $x_0\in a(S)$
  such that $h(x_0)< 2\sigma(S)$.  If $x\in S$, then $d(x,x_0)\le
  4\sigma(S)$, so $h(x)\le h(x_0)+d(x,x_0)< 6\sigma(S)$, as desired.

  Finally, we prove property 3.  Suppose that $S$ and $T$ neighbor
  each other and let $x\in S\cap T$.  By property 2, $\sigma(S)\le
  h(x)\le 6\sigma(T)$ and likewise $\sigma(T)\le h(x)\le 6\sigma(S)$.
  Indeed, since $S$ and $T$ are dyadic squares, we must have
  $\sigma(T)\le 4\sigma(S)\le 16\sigma(T)$, so each square in $U$ can
  be neighbors with at most 4 other squares on each side, for a total
  of 16.
\end{proof}

As a corollary, we obtain:
\begin{cor}\label{cor:adaptive}
  Let $t=2^k$, $k\ge 0$, let $D^2(t)=[0,t]\times[0,t]$, and let
  $h:\sqD(t)\to \R$ be a $1$-Lipschitz function such that $h(x)\ge 1$
  for all $x$.  There is a
  triangulation $\tau_h$ of $\sqD(t)$ such that
  \begin{enumerate}
  \item All vertices of $\tau_h$ are lattice points, and $\tau_h$
    contains no more than $2 t^2$ triangles.
  \item If $x$ and $y$ are connected by an edge of $\tau_h$, then 
    $$\min\{\frac{h(x)}{6},\frac{t}{2}\}\le d(x,y) \le\sqrt{2}h(x).$$
  \item If we consider $\tau_h^{(2)}$ to be the set of triangles of
    $\tau_h$, then 
    $$\sum_{\Delta\in \tau_h^{(2)}} \diam(\Delta)^2 \le 64 t^2.$$
    Likewise, if $\tau_h^{(1)}$ is the set of edges, then 
    $$\sum_{e\in \tau_h^{(1)}} \ell(e)^2 \le 128 t^2.$$
  \end{enumerate}
\end{cor}
\begin{proof}
  Let $U$ be the partition into squares constructed in
  Lemma~\ref{lem:adaptiveSquares}.  Two adjacent squares in $U$ need
  not intersect along an entire edge, so $U$ is generally not a
  polyhedron.  To fix this, we subdivide the edges of each square so
  that two distinct polygons in $U$ intersect either in a vertex, in
  an edge, or not at all; call the resulting polyhedron $U'$.  By
  replacing each $n$-gon in $U'$ with $n-2$ triangles, we obtain a
  triangulation, which we denote $\tau_h$.  We claim that this
  $\tau_h$ satisfies the required properties.

  The first property is clear; the vertices of any dyadic square are
  lattice points by definition, and the area of any triangle whose
  vertices are lattice points is at least $1/2$ by Pick's Theorem.

  The second property follows from the corresponding property of $U$.

  The third property follows from the fact that the number of
  neighbors of each square is bounded.  Since we divide each edge in
  $U$ into at most 4 edges of $U'$, each square $S$ of $U$ corresponds
  to at most $16$ triangles of $\tau_h$, each with diameter at most
  $2\sigma(S)$.  So if
  $\tau_h^{(2)}$ is the set of triangles of $\tau_h$, then
  $$\sum_{\Delta\in \tau_h^{(2)}} \diam(\Delta)^2 \le \sum_{S\in U}
  16(2 \sigma(S))^2.$$ Since $\sum_{S\in U} \sigma(S)^2=\area D^2(t)$,
  this is at most $64 t^2.$ Likewise, if $\tau_h^{(1)}$ is the set of
  edges, then
  $$\sum_{e\in \tau_h^{(1)}} \ell(e)^2 \le \sum_{S\in U}
  32(2 \sigma(S))^2= 128 t^2.$$
\end{proof}

We use this lemma to prove Lemma~\ref{lem:templateExist} by letting
$h(x)\sim r(x)$.

\begin{proof}[Proof of Lemma~\ref{lem:templateExist}]  
  Let $w(i)=w_1\dots w_i$.
  Let $\alpha:[0,\ell]\to \cE$ be the curve corresponding to $w,$
  parameterized so that $\alpha(i)=[w(i)]_\cE$.  If $c_\Sigma$ is the
  maximum length of a curve corresponding to a generator, then $\alpha$ is $c_\Sigma$-Lipschitz.
  Let $t=2^k$ be the smallest power of 2 larger than $\ell$, and let
  $\alpha':[0,t]\to\cE$
  $$\alpha'(x)=\begin{cases}\alpha(x) & \text{if $x\le \ell$} \\
    [I]_{\cE} & \text{otherwise}.
  \end{cases}$$
  Since $\cE$ is non-positively curved, we can use geodesics to fill
  $\alpha'$.  If $x,y\in \cE$, let $\gamma_{x,y}:[0,1]\to \cE$ be a
  geodesic parameterized so that $\gamma_{x,y}(0)=x$,
  $\gamma_{x,y}(1)=y$, and $\gamma_{x,y}$ has constant speed.  We can
  define a homotopy $f:[0,t] \times [0,t]\to \cE$ by
  $$f(x,y)=\gamma_{\alpha'(x),\alpha'(0)}(y/t);$$
  this sends three sides of $D:=[0,t]\times [0,t]$ to $[I]_\cE$ and is a filling of
  $\alpha$.  Since $\cE$ is non-positively curved, this map is
  $2c_\Sigma$-Lipschitz and has area $O(\ell^2)$.
  
  Let $h:D\to \R$,
  $$h(x)=\max\{1,\frac{r(f(x))}{16p^2 c_\Sigma}\}$$
  This function is $1$-Lipschitz.  If $h(x)$ is sufficiently large and
  $B\in D$ is a disc of radius $2h(x)$ around $x$, then $f(B)$ is
  contained in a ball of radius $r(f(x))/(4p^2)$ around $f(x)$.  By
  Cor.~\ref{cor:parabolicNbhds}, $\rho(f(B))$ is contained in a coset
  of a maximal parabolic subgroup.  Let $\tau_h$ be the triangulation
  of $D$ constructed in Cor.~\ref{cor:adaptive}.

  If $v$ is an interior vertex of $\tau_h$, label it $\rho(f(v))$.  If
  $(i,0)$ is a boundary vertex on the side of $D$ corresponding to
  $\alpha'$ and $i\le \ell$, label it by $w(i)$.  Label all the rest
  of the boundary vertices by $I$.  Note that for all vertices $v$, if $g$ is
  the label of $v$, then $f(v)\in g\cS$.

  If $x$ is a lattice point on the boundary of $D$, then
  $f(x)=[I]_\cM$ and so $h(x)=1$.  In particular, each lattice point
  on the boundary of $D$ is a vertex of $\tau_h$, so the
  boundary of $\tau_h$ is a $4t$-gon with vertices labeled
  $I,w(1),\dots,w(n-1),I\dots,I$.  We identify vertices
  labeled $I$ and remove self-edges to get a template $\tau$ for $w$.

  First, property \ref{lem:templateExist:length} follows from
  Lemma~\ref{lem:depthLength}.  That is, if $v_1$ and $v_2$ are the
  endpoints of an edge of $\tau$, labeled by $g_1$ and $g_2$, then
  $d(v_1,v_2)\sim r(f(v_1))\sim r(f(v_2))$, so by
  Lemma~\ref{lem:depthLength},
  $$d_{\Gamma}(g_1,g_2)=O(d(v_1,v_2)+r(f(v_1))+r(f(v_2)))=O(d(v_1,v_2))$$
  as desired.

  Second, note that if $x_1$, $x_2$, and $x_3$ are the vertices of a
  triangle of $\tau$, with labels $g_1$, $g_2$, and $g_3$, then
  Cor.~\ref{cor:adaptive} implies that $\diam\{x_1,x_2,x_3\}\le 2
  h(x_1)$.  If $h(x_1)$ is sufficiently large, then
  Cor.~\ref{cor:parabolicNbhds} shows that $g_1,g_2$, and $g_3$ are in
  the same coset of $U(j,p-j)$ for some $j$; otherwise, by property
  \ref{lem:templateExist:length}, $g_1$, $g_2$, and $g_3$ must be
  within bounded distance of one another.

  Finally, property \ref{lem:templateExist:area} follows from property
  \ref{lem:templateExist:length} and the corresponding property of $\tau_h$.
\end{proof}
Note that it is not necessary that $q\ge 5$ for this template to
exist.  In fact, a suitable generalization of the proposition should hold
for any lattice in a semisimple Lie group.

\section{Shortcuts and normal forms}\label{sec:normalForms}
As before, we let $\Gamma=\SL(p;\Z)$, with the generating set $\Sigma$
consisting of the unit transvections $e_{ij}=e_{ij}(1)$ and the
diagonal matrices.  In this section, we will define a normal form
$\omega:\Gamma\to \Sigma^*$ which associates each element of $\Gamma$
with a word in $\Gamma$ which represents it.  This normal form will
use short representatives of unipotent elements like those constructed
by Lubotzsky, Mozes, and Raghunathan \cite{LMRComptes}.

\subsection{Shortcuts}
Recall that $e_{ij}(x)$, $i\ne j$ represents the matrix which is
obtained from the identity matrix by replacing the $(i,j)$-entry with
$x$.  Lubotzsky, Mozes, and Raghunathan noted that when $p\ge 3$,
this group element can be represented by a word $\short{e}_{ij}(x)$ of
length $\sim \log |x|$, which we call a \emph{shortcut}.  Since the
particular generators used to construct a shortcut will be
important later on, we will define many different ways to shorten a
given transvection: if $S\subset \{1,\dots,p\}$ is a set such that
$i\in S$, $j\not\in S$, and $\#S\ge 2$, then
$\short{e}_{ij;S}(x)$ will be a shortcut for $e_{ij}(x)$ which is a
product of unit transvections lying in the parabolic subgroup
$U(S,\{j\})$.  More generally, recall that if $S, T\subset \{1,\dots,p\}$ are
disjoint and $V\in \R^S\otimes \R^T$, then $u(V)$ represents the
unipotent matrix in $U(S,T)$ corresponding to $V$.  If $\#S\ge 2$, we
will define a curve $\short{u}_S(V)$ which lies in a thick part of
$G$, goes from $I$ to $u(V)$, and has length $O(\log \|V\|_2)$.

We will provide a condensed version of the constructions of
$\short{e}_{ij;S}(x)$ and $\short{u}_S(V)$; for more details, see
\cite{LMRComptes} or \cite{RileyNav}.  We start by defining a solvable
subgroup $H_{S,T}\subset U(S,T)$ for each pair of disjoint sets $S,
T\subset \{1,\dots,p\}$.  Without loss of generality, we may take
$S=\{1,\dots,s\}$ and $T=\{s+1,\dots, s+t\}$.  Let $A$ and $B$ be
$\R$-split, $\Q$-anisotropic tori in $\SL(S)$ and $\SL(T)$
respectively; their integer points are isomorphic to $\Z^{s-1}$ and
$\Z^{t-1}$ respectively.  Let
$$H_{S,T}:=(A\times B)\ltimes (\R^S\otimes \R^T),$$
where $A$ acts on $\R^S$ on the left and $B$ acts on $\R^T$ on the
right.  Without loss of generality, we may take $S=\{1,\dots,s\}$ and
$T=\{s+1,\dots, s+t\}$, and write
$$H_{S,T}= \left\{\begin{pmatrix}M & V & 0\\
  0 & N & 0\\
  0 & 0 & I
\end{pmatrix}\middle| M\in A, N\in B, V\in \R^S\otimes \R^T\right
\}.$$ Note that the integer points of $H_{S,T}$ form a cocompact
lattice in $H_{S,T}$, so $H_{S,T}$ lies in a thick part of $\SL(p).$
We may conjugate $H_{S,T}$ so that $A$ and
$B$ become the subgroups of diagonal matrices with positive
coefficients; it follows that $\R^S\otimes \R^T$ is exponentially
distorted in $H_{S,T}$ as long as $\#S\ge 2$ or $\#T\ge 2$.

Thus, given any $V\in \R^S\otimes \R^T$, there is a curve in $H_{S,T}$ from $I$ to
$u(V)$ of length $\sim \log \|V\|_2$.
Note that this curve may depend on $S$ and $T$ as well as $V$.  The
following construction removes the dependence on $T$:
If $\#S\ge 2$, $S^c$ is the complement of $S$,
$V\in \R^S\otimes \R^{S^c}$, and $\{z_1,\dots, z_p\}$ is the standard
basis of $\R^p$, we can write $V=\sum_{i\in S^c} v_i \otimes z_i$ for
some vectors $v_i\in \R^S$.  For all $i\in S^c$, define $u_i:[0,1]\to
H_{S,S^c}$ to be a geodesic in $H_{S,\{i\}}$ which connects $I$ to
$u(v_i \otimes z_i)$ and define
$$\short{u}_S(V)=\prod_{i\in S^c} u_i.$$
This is a curve connecting $I$ to $u(V)$ which has length $O(\log
\|V\|_2)$.  Furthermore, if $T\subset \{1,p\}$ is such that $V\in
\R^S\otimes \R^T$, then $\short{u}_S(V)$ is a curve in $H_{S,T}$.  We
think of it as the result of using a torus in $\SL(S)$ to ``compress''
$u(V)$.

If $\#S\ge 2$, $i\in S$, $j\not\in S$, and $x\in \Z$, then
$\short{u}_S(x z_i\otimes z_j)$ is a curve in $H_{S,\{j\}}\subset
U(S,\{j\})$ which connects $I$ to $e_{ij}(x)$.  Let
$\short{e}_{ij;S}(x)$ be a word in $U(S,\{j\})$ approximating
$\short{u}_S(x z_i\otimes z_j)$.  Since $H_{S,\{j\}}$ lies in a thick
part of $U(S,\{j\})$, the length of $\short{e}_{ij;S}(x)$ is
comparable to the length of $\short{u}_{S}(xz_i\otimes z_j)$.  

In many cases, the precise value of $S$ does not matter, so for each
pair $(i,j)$, we choose a $d_{ij}$ such that $d_{ij}\not\in \{i,j\}$
and define $\short{e}_{ij}(x)=\short{e}_{ij;\{i,d_{ij}\}}(x)$.  As a
special case, for all $i,j$, we set $\short{e}_{ij}(\pm 1)=e_{ij}^{\pm 1}$.

\subsection{A normal form}
Recall that if $H\subset \Gamma$, we say that  $w$ is a \emph{shortcut word} in $H$ if we can
write $w=\prod_{i=1}^n w_i$, where each $w_i$ is either a diagonal
matrix in $H$ or a shortcut $\short{e}_{a_ib_i}(x_i)$ where
$e_{a_ib_i}(x_i)\in H$.  Note that any word in $H$ is automatically a
shortcut word, but not vice versa, since a shortcut for an element of
$H$ may use generators that are not in $H$.

We claim:
\begin{lemma}\label{lem:normalForm}
  There is a normal form $\omega:\Gamma\to \Sigma^*$ such that:
  \begin{enumerate}
  \item For all $g\in \Gamma$, $\ellw(\omega(g))=O(d_\Gamma(I,g))$.
  \item For all $i,j\in \{1,\dots, p\}$ with $i\ne j$ and $x\in \Z$, 
    $$\omega(e_{ij}(x))=\short{e}_{ij}(x).$$
  \item If $g\in P$ where $P=U(S_1,\dots,S_k)$ is a group of block
    upper-triangular matrices, then $\omega(g)$ is a product of a
    bounded number of shortcut words in the diagonal blocks of $P$,
    a bounded number of shortcuts corresponding to off-diagonal
    entries of $P$, and possibly one diagonal matrix.
% kind of want a case for subgroups SL(q).
  \end{enumerate}
\end{lemma}

\begin{proof}
  If $g=e_{ij}(x)$, we define $\omega(g)=\short{e}_{ij}(x)$; this
  satisfies all three conditions.  

  Otherwise, let $g\in \Gamma$ and let $P=U(S_1,\dots,S_k)\in \cP$ be
  the unique minimal $P\in \cP$ containing $g$.  Then $g$ is a
  block-upper-triangular matrix which can be written as a product
  \begin{equation}\label{eq:blockDecomp}
    g=\begin{pmatrix}
      m_1   & V_{12} & \dots  & V_{1k} \\
      0     & m_2     & \dots  & V_{2k} \\
      \vdots & \vdots  & \ddots & \vdots  \\
      0     & 0       & \dots  & m_k
    \end{pmatrix} d,
  \end{equation}
  where the $i$th block of the matrix corresponds to $S_i$.  Here,
  $V_{i,j}\in \Z^{S_i}\otimes \Z^{S_j}$, and $d\in D$ is a diagonal
  matrix chosen so that $\det m_i=1$.  If $P=\Gamma$, then there is only
  one block, and we take $m_1=g$, and $d=I$.  We can write $g$ as a product:
  $$\gamma_i:=\biggl(\prod_{j=1}^{i-1} u(V_{ji})\biggr) m_i$$
  $$g=\gamma_k\dots \gamma_1 d.$$
  We will construct $\omega(g)$ by replacing the terms in this product
  decomposition with shortcut words.

  First, consider the $m_k$.  When $\#S_k\ge 3$, we can use
  Thm.~\ref{thm:LMR} to replace $m_k$ by a word in $\SL(S_k;\Z)$, but the theorem does not apply when $\#S_k=2$ because
  $\SL(2;\Z)$ is exponentially distorted inside $\SL(p;\Z)$.  We thus
  use a variant of the Lubotzky-Mozes-Raghunathan theorem to write $m_k$
  as a shortcut word $\SL(2;\Z)$.
  \begin{prop}[{cf.\ \cite{LMRComptes}}]\label{prop:shortLMR}
    There is a constant $c$ such that for all $g\in \SL(k;\Z)$, there is
    a shortcut word in $\SL(k;\Z)$ which represents $g$ and has length
    $$\ellw(w)\le c\log \|g\|_2.$$
  \end{prop}

  For $i=1,\dots,k$, let $\short{m}_i$ be a shortcut word
  representing $m_i$ as in Prop.~\ref{prop:shortLMR}.  For $1\le i< j\le
  k$ and $V\in \Z^{S_i}\otimes \Z^{S_j}$, let
  $$\short{n}(V):=\prod_{a\in S_i, b\in
    S_j}\short{e}_{ab}(x_{ab}),$$
  where $x_{ab}$ is the $(a,b)$-coefficient of $V$; this is a shortcut
  word representing $u(V)$.

  Let
  $$\short{\gamma}_i:=\biggl(\prod_{j=1}^{i-1} \short{n}(V_{ji})\biggr) \short{m}_i$$
  $$\omega(g)=\short{\gamma}_k\dots \short{\gamma}_1d.$$
  This is a word in $\SL(p;\Z)$ which represents $g$.  It is
  straightforward to show that there is a constant $c_\omega$
  independent of $g$ such that $\ellw(\omega(g))\le c_\omega
  d_\Gamma(I,g)$.

  Furthermore, if $Q=U(T_1,\dots,T_r)$ is a block upper-triangular
  subgroup and $g\in Q$, then $Q\supset P$ and for every $i$, there is
  a $j$ such that $S_i\subset T_j$.  In particular, $\short{m}_i$ is a
  shortcut word in $\SL(T_j;\Z)$, and each shortcut making up
  $\short{V}_{ji}$ is either contained in $\SL(T_a)$ for some $a$ or
  corresponds to an off-diagonal entry of $Q$.  Since there are a
  bounded number of $\short{m}_i$ and a bounded number of terms in the
  $\short{V}_{ji}$, this normal form satisfies property 3.
\end{proof}

\subsection{Summary of notation}
For quick reference, we summarize the above constructions, which will
be used throughout the rest of the paper:
\begin{itemize}
\item $u(V)$ is the unipotent element in $U(S,T)$ corresponding to
  $V$.
\item $H_{S,T}$ is a solvable subgroup of $U(S,T)$, isomorphic to 
  $$(\R^{\#S-1}\times \R^{\#T-1})\ltimes (\R^S\otimes \R^T).$$
\item $\short{u}_S(V)$ is a curve in $H_{S,T}$ which connects $I$ and $u(V)$
\item If $i\in S$ and $j\not\in S$, then $\short{e}_{ij;S}(x)$ is a
  word in $U(S,\{j\})$ which represents $e_{ij}(x)$
\item For each $1\le i\ne j\le p$, we choose some $d_{ij}\not\in
  \{i,j\}$ and define
  $\short{e}_{ij}(x):=\short{e}_{ij;\{i,d_{ij}\}}(x)$.
\end{itemize}

\section{Manipulating shortcuts}\label{sec:shortManip}

In Section~\ref{sec:normalForms}, we constructed shortcuts $\short{e}_{ij;S}(x)$,
that is, words of logarithmic length which represent transvections
$e_{ij}(x)$.  The normal form $\omega$ is built using products of
these shortcuts, and in this section, we will develop ways to
manipulate such products.

\subsection{Moving shortcuts between solvable groups}

One of the main ideas behind these tools is that when $p$ is large, we
can construct shortcuts for $e_{ij}(x)$ which lie in small subgroups
of $\Gamma$ and we can construct quadratic-area homotopies from one to
another.  

This subsection is devoted to proving that when $p$ is large,
shortcuts which come from different solvable subgroups can be
connected by quadratic-area homotopies.  We will prove the following
lemma:
\begin{lem}\label{lem:shortEquiv}  
  If $p\ge 5$ and if $S\subset \{1,\dots, p\}$ is such that $2\le
  \#S \le p-2$, $i\in S$ and $j\not\in S$, then 
  $$\delta_{\Gamma}(\short{e}_{ij}(x),
  \short{e}_{ij;S}(x))=O((\log |x|)^2).$$
\end{lem}

We will use a
special case of a theorem of Leuzinger and Pittet on Dehn
functions of solvable groups.  Recall that the curves
$\short{u}_S(V)$ used to define the $\short{e}_{ij}$'s lie in
solvable subgroups of the form
$$H_{S,T}=(A\times B)\ltimes (\R^S\otimes \R^T),$$
where $A$ and $B$ are $\R$-split, $\Q$-anisotropic tori in $\SL(S)$ and $\SL(T)$
respectively.  These subgroups are contained in the thick part of $G$, and when
either $S$ or $T$ is large enough, results of Leuzinger and Pittet
\cite{LeuzPitQuad} imply that $H_{S,T}$ has
quadratic Dehn function (see also \cite{deCorTess}): % check references here;
% there might be better ones.
\begin{thm}\label{thm:HDehn}
  If $s=\#S\ge 3$ or $t=\#T\ge 3$, then $H_{S,T}$ has a quadratic Dehn
  function.
\end{thm}

We use manipulations in $H_{S,T}$ to prove the following lemma:
\begin{lem} \label{lem:xiConj} 
  Let $S\subset \{1,\dots,p\}$ be such that $2\le \#S\le p-2$ and let $T=S^c$ be the complement of
  $S$.  Let $0<\epsilon<1/2$ be sufficiently
  small that $H_{S,T}\subset G(\epsilon)$.  If $\gamma$ is a curve in the $\epsilon$-thick part
  of $\SL(S)\times \SL(T)$ which connects $(I,I)$ to $(M,N)$, and if
  $V\in \R^S\otimes \R^{T}$, then
  $$\delta_{G(\epsilon')}(\gamma \short{u}_S(V)
  \gamma^{-1},\short{u}_S(MVN^{-1}))= O((\ellw(\gamma)+\log
  (\|V\|_2+2))^2).$$
  where $\epsilon'$ is independent of $\gamma$ and $V$.
\end{lem}
\begin{proof}
  Let $s=\#S$ and $t=\#T$.  Let $A$ and $B$ be the tori used to define
  $H_{S,T}$ and let Let $\{v_1,\dots,v_s\}\subset \R^S$ and
  $\{w_1,\dots,w_t\}\subset \R^T$ be the corresponding eigenbases of
  $\R^S$ and $\R^T$.

  We first consider the case that $\gamma$ is a curve in
  $\SL(T)$ and that $V=x v_i\otimes w_j$.  In this case $M=I$ and we want to fill the curve
  $$\omega:=\gamma \short{u}_S(V) \gamma^{-1}\short{u}_S(VN^{-1})^{-1}.$$
  Let $\ell=\ellc(\omega)$.
  Let $\delta=\exp(-\ellw(\gamma))$ and let $D\in A$ be such
  that $\|x D v_i\|_2\le \delta$; we can choose $D$ so that
  $d_{A}(I,D)=O(\ell)$.  If we conjugate $\omega$ by
  $D$, it becomes easy to fill.  We will fill $\omega$
  with a disc of the form shown in Figure~\ref{fig:hstfill}.

  \begin{figure}
    \def\svgwidth{3.5in}
    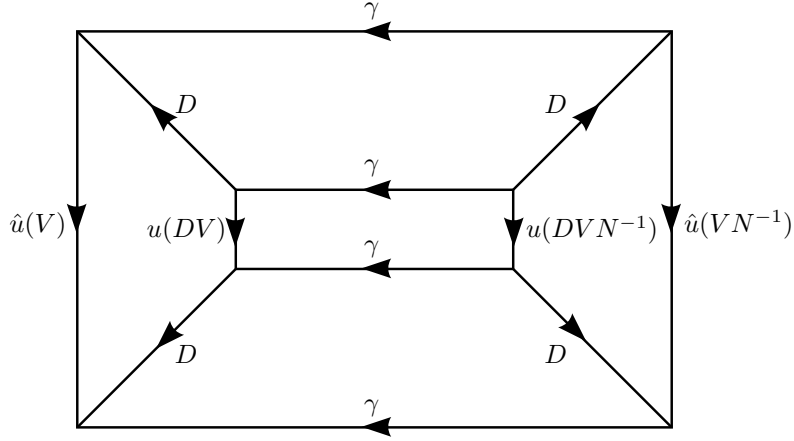
    \caption{\label{fig:hstfill} A filling of $\omega=\gamma \short{u}_S(V) \gamma^{-1}\short{u}_S(VN^{-1})^{-1}$.}
  \end{figure}

  This disc is comprised of four trapezoids and a central ``thin
  rectangle''.  Each of the edges labeled $D$ corresponds to a
  translate of the geodesic in $A$ which connects $I$ to $D$,
  and each edge has length at most $O(\ell)$.  Each trapezoid
  can be filled with quadratic area.  The left and right trapezoids
  are contained in $H_{S,T}$, so they have quadratic filling area in
  $H_{S,T}$; furthermore, since $H_{S,T}$ lies in the thick part of
  $G$, the filling stays in the thick part.  The top and bottom
  trapezoids each represent a commutator of a curve in $\SL(S)$ and a
  curve in $\SL(T)$, and so can be filled by the rectangle resulting
  from the product of those curves.  Each curve stays in some thick
  part of $\SL(S)$ or $\SL(T)$, so the rectangle does as well.

  The central thin rectangle can be filled by a disc of area
  $O(\ell)$.  Call the edges labeled by $\gamma$ the ``long
  edges'' of the rectangle.  Since $\delta$ is small, these
  long edges synchronously fellow travel, and since they lie in the
  thick part they can be filled by a
  disc in the thick part of area $O(\ell)$.  This gives a
  quadratic filling of $\omega$.

  The same technique works if instead we have $\gamma:[0,1]\to \SL(S)$
  and $V=x v_i\otimes w_j$.  The main change is that $D$ is now a
  matrix in $B$ such that $\|x w_j D\|_2\le \delta$.

  Now suppose $\gamma$ is a curve in $\SL(S)\times \SL(T)$.  It
  can be homotoped to a concatenation of curves
  $\gamma=\gamma_S\gamma_T$, where $\gamma_S$ and $\gamma_T$ are the
  projections of $\gamma$ to each factor.  This homotopy can be taken
  to have quadratic area and lie in the thick part of $G$, and the
  lemma can be applied to $\gamma_S$ and $\gamma_T$ separately.  This
  proves the lemma in the case that $V=x v_i\otimes w_j$.

  In general, we can decompose $V$ as a sum of eigenvectors
  $V=\sum_{i,j} x_{ij} v_i\otimes w_j$, so we will use the quadratic
  Dehn function of $H_{S,T}$ to break $\short{u}_S(V)$ up into pieces
  corresponding to each eigenvector and apply the lemma to each piece.
  We can construct a homotopy from
  $\gamma \short{u}_S(V)\gamma^{-1}$ to $\short{u}_S(MVN^{-1})$ which
  goes through the following stages:
  \begin{align*}
    &\gamma \short{u}_S(V)\gamma^{-1} & \\
&  \gamma \biggl(\prod_{i,j} \short{u}_S(x_{ij} v_i\otimes
    w_j)\biggr)\gamma^{-1} & & \text{by Thm.~\ref{thm:HDehn}} \\
    &\prod_{i,j} \gamma \short{u}_S(x_{ij} v_i\otimes w_j)\gamma ^{-1}
     && \text{by free insertions} \\
    &\prod_{i,j} \short{u}_S(M(x_{ij} v_i\otimes w_j)N^{-1}) && 
    \text{by the arguments above}\\
    &    \short{u}_S(MVN^{-1}) &  &\text{by Thm.~\ref{thm:HDehn}}
  \end{align*}
  Each stage has quadratic area, so the homotopy as a whole has
  quadratic area.
\end{proof}

Let  $\{z_1,\dots,z_p\}$ be the standard basis of $\R^p$.
Lemma~\ref{lem:shortEquiv} then follows from the following lemma:
\begin{lem}\label{lem:fullShortEquiv}
  There is an $0<\epsilon<1/2$ such that if $i\in S,S'$ and $j\not\in
  S\cup S'$, where $2\le \#S,\#S'\le p-2$, and if $x\in\R$, then
  $$\delta_{G(\epsilon)}(\short{u}_S(x z_i\otimes z_j),\short{u}_{S'}(x
  z_i\otimes z_j))=O((\log |x|)^2).$$

  In particular, 
  $$\delta_{\Gamma}(\short{e}_{ij}(x),
  \short{e}_{ij;S}(x))=O((\log |x|)^2).$$
\end{lem}
\begin{proof}
  First, consider the case that $S\subset S'$.

  Let $T=S^c$
  be the complement of $S$ and $T'=(S')^c$ be the complement of $S'$.
  Then $\short{u}_S(x z_i\otimes z_j)$ is a curve in
  $H_{S,T}$, and $H_{S,T}$ and $H_{S',T'}$ both have quadratic Dehn functions.  Let
  $s=\#S$, $t=\#T$.
  
  Recall that $A=\R^{s-1}\subset H_{S,T}$ 
  is an $\R$-split, $\Q$-anisotropic torus in $\SL(S)$; let $A(\Z)$ be 
  the integer points of $A$.  We can decompose $x z_i$ as a sum
  of eigenvectors $x z_i=\sum_k v_k$ and ``compress'' each term in
  this sum using $A$.  That is, there are vectors $y_k\in \R^S$ such
  that $\|y_k\|_2\le 1$ and elements $A_k\in A(\Z)$ such that
  $v_k=A_ky_k$ and $d_{A}(I,A_k)=O(\log |x|)$.
  Then
  $$e_{ij}(x)=\prod_k A_k u(y_k\otimes z_j)A^{-1}_k.$$
  Let $\gamma_k$, $k=1,\dots, s$ be a geodesic in $A$ which connects
  $I$ to $A_k$ and has length $O(\log |x|)$.  Let $\U_k:[0,1]\to G$
  be the curve $\U_k(t)=u(t y_k\otimes z_j)$.  We can then construct a
  curve
  $$\omega=\prod_k \gamma_k \U_k \gamma_k^{-1}$$
  which connects $I$ to $e_{ij}(x)$.

  We can use $\omega$ as an intermediate stage in a homotopy between
  $\short{u}_S(x z_i\otimes z_j)$ and $\short{u}_{S'}(x z_i\otimes
  z_j)$.  On one hand, $\omega$ lies in $H_{S,T}$ and has length
  $O(\log |x|)$, so there is a quadratic-area homotopy from
  $\short{u}_S(x z_i\otimes z_j)$ to $\omega$.  On the other hand,
  since $\gamma_k$ lies in a thick part of $\SL(S)$, $\omega$ also lies in a
  thick part of $\SL(S')$, and we can apply Lemma~\ref{lem:xiConj} to
  each term of $\omega$ to construct a homotopy from $\omega$ to
  $$\omega'=\prod_k \short{u}_{S'}(v_k \otimes z_j).$$
  This is a curve in $H_{S',T'}$ of length $O(\log |x|)$, so there is
  a quadratic-area homotopy from $\omega'$ to $\short{u}_{S'}(x z_i\otimes z_j)$.
  Concatenating these homotopies produces a homotopy from $\short{u}_{S}(x
  z_i\otimes z_j)$ to $\short{u}_{S'}(x z_i\otimes z_j)$ as desired.
  
  For the general case, let $k\in S$ and $k'\in S'$ be such that $i\ne
  k$, $i\ne k'$.  Then, if $V=x z_i\otimes z_j$, we can use the
  argument above to construct a homotopy from $\short{u}_{S}(V)$ to
  $\short{u}_{S'}(V)$ which goes through the stages
  $$\short{u}_{S}(V)\to \short{u}_{\{i,k\}}(V)\to \short{u}_{\{i,k,k'\}}(V)\to \short{u}_{\{i,k'\}}(V)\to \short{u}_{S'}(V)$$
  (if $k=k'$, the middle stages can be omitted).
 
  Since $\Gamma$ acts geometrically on $G(\epsilon)$ and
  $\short{e}_{ij;S}(x)$ is an approximation of $\short{u}_S(x
  z_i\otimes z_j)$, this implies that 
  $$\delta_{\Gamma}(\short{e}_{ij;S}(x),
  \short{e}_{ij;S'}(x))=O((\log |x|)^2),$$
  and in particular, 
  $$\delta_{\Gamma}(\short{e}_{ij}(x),
  \short{e}_{ij;S}(x))=O((\log |x|)^2).$$
\end{proof}

We also note the following corollary which will be useful later:
\begin{cor}\label{cor:uShortEquiv}
  If $S,S',T \subset \{1,\dots p\}$ are such that $2\le \#S,\#S'\le
  p-2$, $S\cap T=\emptyset$, and $S'\cap T=\emptyset$, and if $V\in
  \R^{S\cap S'}\otimes \R^T$, then there is an $0<\epsilon<1/2$ such
  that 
  $$\delta_{G(\epsilon)}(\short{u}_S(V),\short{u}_{S'}(V))=O((\log \|V\|_2)^2).$$
\end{cor}
\begin{proof}
  For all $i\in S\cap S'$ and $j\in T$, let $v_{ij}$ be the
  coefficient of $V$ in the $(i,j)$-position.  Let
  $$\omega=\prod_{i,j} \short{u}_S(v_{ij} z_i\otimes z_j)$$
  $$\omega'=\prod_{i,j} \short{u}_{S'}(v_{ij} z_i\otimes z_j).$$
  Then $\omega$ is a curve in $H_{S,S^c}$ with length $O(\log
  \|V\|_2)$, so there is a quadratic-area homotopy from
  $\short{u}_S(V)$ to $\omega$ and likewise from $\omega'$ to
  $\short{u}_{S'}(V)$.  By Lemma~\ref{lem:fullShortEquiv}, there is a
  quadratic-area homotopy from $\omega$ to $\omega'$.  Combining these
  homotopies proves the corollary.
\end{proof}

Lemma~\ref{lem:relToTri} is then a corollary of Lemma~\ref{lem:shortEquiv}:
\begin{proof}[Proof of Lemma~\ref{lem:relToTri}]\label{sec:relToTriProof}
Suppose that $w$ is a shortcut word in $\SL(q;\Z)$ and $q\ge 5$.  We
can write $w=\prod_{i=1}^n w_i$, where each $w_i$ is either a diagonal
matrix in $H$ or a shortcut $\short{e}_{a_ib_i}(x_i)$ where
$e_{a_ib_i}(x_i)\in \SL(q;\Z)$.  Since $q\ge 3$, we can use
Lemma~\ref{lem:shortEquiv} to replace each shortcut
$\short{e}_{a_ib_i}(x_i)$ by a shortcut $\short{e}_{a_ib_i;S_i}(x_i)$
which lies in $\SL(q;\Z)$ at a total cost of order $O(\ellw(w)^2)$.
The result is a word $w'$ in $\SL(q;\Z)$ of length $O(\ellw(w))$, and 
$\delta(w,w')=O(\ell(w)^2)$ as desired.
\end{proof}

\subsection{The shortened Steinberg presentation}

The Steinberg presentation gives relations between products of
elementary matrices; in this section, we will develop ways to
manipulate the corresponding shortcut words.  

This subsection is devoted to building an analogue of the Steinberg
presentation for shortcut words.  We will prove the following lemma:
\begin{lem}[The shortened Steinberg presentation]\label{lem:infPres} If $x,y\in \Z\setminus
  \{0\}$, then
  \begin{enumerate}
  \item \label{lem:infPres:add}
    If $1\le i,j\le p$ and $i\ne j$, then
    $$\delta_{\Gamma}(\short{e}_{ij}(x)\short{e}_{ij}(y),\short{e}_{ij}(x+y))=O((\log |x|+\log |y|)^2).$$
    In particular,
    $$\delta_{\Gamma}(\short{e}_{ij}(x)\short{e}_{ij}(-x))=\delta_{\Gamma}(\short{e}_{ij}(x)^{-1},\short{e}_{ij}(-x))=O((\log |x|)^2).$$
  \item \label{lem:infPres:multiply}
    If $1\le i,j,k\le p$ and $i\ne j\ne k$, then
    $$\delta_{\Gamma}([\short{e}_{ij}(x),\short{e}_{jk}(y)],\short{e}_{ik}(xy))= O((\log |x|+\log |y|)^2).$$
  \item \label{lem:infPres:commute}
    If $1\le i,j,k,l\le p$, $i\ne l$, and $j\ne k$
    $$\delta_{\Gamma}([\short{e}_{ij}(x),\short{e}_{kl}(y)])=O((\log |x|+\log |y|)^2).$$
  \item \label{lem:infPres:swap}
    Let $1\le i,j,k,l\le p$, $i\ne j$, and $k\ne l$, and 
    $$s_{ij}=e_{ji}^{-1}e_{ij}e_{ji}^{-1},$$
    so that $s_{ij}$ represents
    $$\begin{pmatrix} 0 & 1 \\ -1 & 0
    \end{pmatrix}\in\SL(\{i,j\};\Z).$$  Then
    $$\delta_{\Gamma}(s_{ij} \short{e}_{kl}(x)
    s^{-1}_{ij},\short{e}_{\sigma(k)\sigma(l)}(\tau(k,l)x))=O( (\log
    |x|+\log |y|)^2),$$
    where $\sigma$ is the permutation switching
    $i$ and $j$, and $\tau(k,l)=-1$ if $k=i$ or $l=i$ and $1$
    otherwise.
  \item \label{lem:infPres:diag} If $b=\diagmat(b_1,\dots,b_p)$, then 
    $$\delta_{\Gamma}(b \short{e}_{ij}(x) b^{-1},\short{e}_{ij}(b_i b_j x))=O( \log |x|^2).$$  \end{enumerate}
\end{lem}

\begin{proof}
  Part \ref{lem:infPres:add} follows from Theorem~\ref{thm:HDehn}.
  Recall that there is a $d_{ij}$ such that
  $\short{e}_{ij}(x)=\short{e}_{ij;\{i,d_{ij}\}}(x)$.  Let
  $S=\{i,d_{ij}\}$.  Then
  $\short{e}_{ij}(x)\short{e}_{ij}(y)\short{e}_{ij}(x+y)^{-1}$ is an
  approximation of a closed curve in $H_{S,S^c}$ of length $O((\log
  |x|+\log |y|)^2).$ Since $H_{S,S^c}$ has quadratic Dehn function, it
  can be filled by a quadratic-area disc.

  The rest of the parts of Lemma~\ref{lem:infPres} involve conjugating
  a shortcut by a word.  Most of the proofs below follow the same
  basic outline: to conjugate $\short{e}_{ij}(x)$ by a word $w$
  representing a matrix $M$, first choose $S$ such that $M\in \SL(S)$, where
  $i\in S$ and $j\not \in S$.  Next, replace $w$ with a word $w'$ in
  $\SL(S)$ and replace $\short{e}_{ij}(x)$ by $\short{e}_{ij;S}(x)$.  Finally, use
  Lemma~\ref{lem:xiConj} to conjugate $\short{e}_{ij;S}(x)$ by $w'$.

  \noindent Part \ref{lem:infPres:multiply}:

  Let $d\not\in \{i,j,k\}$ and let $S=\{i,j,d\}$, so that
  $\short{e}_{ij;\{i,d\}}(x)$ is a word in $\SL(S;\Z)$.  We construct
  a homotopy going through the stages
  \begin{align*}
    \omega_0&=  [\short{e}_{ij}(x),\short{e}_{jk}(y)]\short{e}_{ik}(xy)^{-1} \\
    \omega_1&=  [\short{e}_{ij;\{i,d\}}(x),\short{u}_S(y z_{j}\otimes z_{k})]\short{e}_{ik;S}(xy)^{-1}\\
    \omega_2&=  \short{u}_S((xy z_i+y z_{j})\otimes z_{k})\short{u}_S(y z_{j}\otimes z_{k})^{-1}\short{u}_S(xy z_{i}\otimes z_k)^{-1}.
  \end{align*}
  Here, we use Lemma~\ref{lem:shortEquiv} to construct a homotopy
  between $\omega_0$ and $\omega_1$.  The homotopy between $\omega_1$
  and $\omega_2$ is an application of Lemma~\ref{lem:xiConj} with
  $\gamma=\short{e}_{ij;\{i,d\},\{j\}}(x)$ and $V=y z_{j}\otimes
  z_{k}$.  Finally, $\omega_2$ is a curve in $H_{S,S^c}$ with length
  $O(\log |x|+\log |y|)$, and thus has filling area $O((\log |x|+\log
  |y|)^2)$.  The total area used is $O((\log |x|+\log |y|)^2)$.
  
  \noindent Part \ref{lem:infPres:commute}:

  Let $S=\{i,k,d\}$.  We use the same
  techniques to construct a homotopy going through the stages
  \begin{align*}
    & [\short{e}_{ij}(x),\short{e}_{kl}(y)]\\
    & [\short{e}_{ij;S}(x),\short{e}_{kl;S}(y)] & & \text{by Lem.~\ref{lem:shortEquiv}}\\
    & \emptyword& & \text{by Thm.~\ref{thm:HDehn} applied to $H_{S,S^c}$},
  \end{align*}
  where $\emptyword$ represents the empty word.
  This homotopy has area $O((\log |x|+\log |y|)^2)$.

  \noindent Part \ref{lem:infPres:swap}:

  We consider several cases depending on $k$ and $l$.  When
  $i,j,k,$ and $l$ are distinct, the result follows from part \ref{lem:infPres:commute},
  since $s_{ij}=e_{ji}^{-1}e_{ij}e_{ji}^{-1}$, and we can use part
  \ref{lem:infPres:commute} to commute each letter past $\short{e}_{kl}(x)$.  If $k=i$ and
  $l\ne j$, let $d\not\in \{i,j,l\}$, and let
  $S=\{i,j,d\}$.  There is a homotopy from
  $$s_{ij} \short{e}_{il}(x) s^{-1}_{ij}\short{e}_{jl}(-x)^{-1}$$
  to
  $$s_{ij} \short{u}_S(x z_i\otimes z_l) s^{-1}_{ij}\short{e}_{jl}(x z_j\otimes z_l)$$
  of area $O( (\log |x|)^2),$ and since $s_{ij}$ is a word in $\SL(S;\Z)$, the
  proposition follows by an application of Lemma \ref{lem:xiConj}.  A
  similar argument applies to the cases $k=j$ and $l\ne i$; $k\ne i$
  and $l= j$; and $k\ne j$ and $l= i$.

  If $i=k$ and $j=l$, let $d\not \in \{i,j\}$.
  There is a homotopy going through the stages
  \begin{align*}
&    s_{ij} \short{e}_{ij}(x) s^{-1}_{ij} & \\
&    s_{ij} [e_{id},\short{e}_{dj}(x)] s^{-1}_{ij}& & \text{ by part (\ref{lem:infPres:multiply})}\\
&    [s_{ij}e_{id}s^{-1}_{ij},s_{ij}\short{e}_{dj}(x)s^{-1}_{ij}]&  &\text{ by free insertion}\\
&    [e_{jd}^{-1},\short{e}_{di}(x)] & & \text{ by previous cases}\\
&    \short{e}_{ji}(-x) & &\text{ by part (\ref{lem:infPres:multiply})}
  \end{align*}
  and this homotopy has area $O( (\log |x|)^2)$.  One can treat the
  case that $i=l$ and $j=k$ the same way.

  Since any diagonal matrix in $\Gamma$ is the product of at most
  $p$ elements $s_{ij}$, part \ref{lem:infPres:diag} follows from part \ref{lem:infPres:swap}.
\end{proof}

\section{Reducing to diagonal blocks}\label{sec:redDiagProof}

In this section, we work to prove Lemma~\ref{lem:redDiag}, which
claims that we can break an $\omega$-triangle in
a block upper-triangular subgroup $U(S_1,\dots,S_k)\subset
\SL(p;\Z)$ into shortcut words in the blocks
$\SL(S_i;\Z)$ on the diagonal.  As before, we let $G=\SL(p;\R)$ and $\Gamma=\SL(p;\Z)$.

Let $P:=U(S_1,\dots,S_k)$ and let $P^+\subset P$ be the
finite-index subgroup consisting of matrices in $P$ whose diagonal
blocks all have determinant 1.  Let $K:=\times_i \SL(S_i;\Z)$
and consider the map $\eta:P^+\to K$ which sends an element of $P^+$ to
its diagonal blocks.  If $N=\ker \eta$, we can write $P^+$ as a
semidirect product $P^+=K\ltimes N$.

In most cases, one can prove Lemma~\ref{lem:redDiag} in two stages:
first, break an $\omega$-triangle in $P^+$ into shortcut words in $K$
and $N$, then fill the resulting shortcut words.  This is harder to do
when $P=U(p-1,1)$ or $U(1,p-1)$, because we can't use
Lemma~\ref{lem:xiConj} to manipulate the shortcut words.  It is
possible to use Lemma~\ref{lem:infPres} to conjugate unipotent
matrices by words in $\SL(p-1)$ one generator at a time, but this
produces a cubic bound on the Dehn function rather than a quadratic
bound.  Instead, we will use the methods of Sec.~\ref{sec:redParaProof}
to break words in $U(p-1,1)$ and $U(1,p-1)$ into $\omega$-triangles in
smaller parabolic subgroups.  In the next subsection, we consider the
case that $\#S_i \le p-2$ for all $i$, and in
Sec.~\ref{subsec:specialCase}, we will consider the case of $U(p-1,1)$
and $U(1,p-1)$.

\subsection{Case 1: Small $S_i$'s}\label{subsec:Case1}

Let $P$, $P^+$, $K$, and $N$ be as above.  

The goal of this section is to prove:
\begin{prop}\label{prop:paraReduceSmall}  
  Let $P=U(S_1,\dots,S_k)$, where $\#S_i\le p-2$ for all
  $i$.  If $g_1,g_2,g_3\in P$ and $g_1g_2g_3=1$, let
  $$w=\omega(g_1)\omega(g_2)\omega(g_3).$$
  Then we can break $w$ into words $v_1,\dots, v_k$ at cost
  $O(\ell(w)^2)$, where $v_i$ is a shortcut word in $\SL(S_i;\Z)$ and
  $\ell(v_i)=O(\ell(w))$.
\end{prop}

We will prove this by breaking $w$ into a product of a shortcut word
in $K$ and a shortcut word in $N$, then filling each of these shortcut
words.  

If $g\in \Gamma$ and $Q=U(T_1,\dots,T_r)$ is the minimal element of
$\cP$ containing $g$, then $\omega(g)$ is a product of a shortcut word
in $\SL(T_i)$ for each $i$, at most $p^2$ shortcuts
$\short{e}_{ij}(x_{ij})$ (one for each entry above the diagonal), and
a diagonal matrix.  If $g\in P$, then $Q\subset P$, so each $T_i$ is a
subset of some $S_j$.  In particular, each shortcut word in $\SL(T_i)$
is also a shortcut word in some $\SL(S_j)$, and each transvection
$e_{ij}(x)$ with $i>j$ is either an element of some
$\SL(S_j)$ or an element of $N$.  Consequently, we can consider
$w$ as a product of at most $3p^2$ shortcut words in the
$\SL(S_i)$, at most $3p^2$ shortcuts $\short{e}_{ij}(x_{ij})$ such
that $e_{ij}(x_{ij})\in N$, and three diagonal matrices.  

% duplicate \short{n}, \short{V}?
If $V\in \Z^{S_i}\otimes \Z^{S_j}$, let
$$\short{n}_{ij}(V):=\prod_{a\in S_i, b\in S_j}\short{e}_{ab}(v_{ab}).$$
Call a shortcut word in one of the $\SL(S_i)$ a \emph{diagonal word}
and call a word of the form $\short{n}_{ab}(V)$ an \emph{off-diagonal
  block}.  If $e_{ij}(x)\in N$, then $\short{e}_{ij}(x)$ is an
off-diagonal block, so $w$ is a product of up to 3 diagonal matrices,
up to $3p^2$ diagonal words, and up to $3p^2$ off-diagonal blocks.

We can use this terminology to describe the proof of Prop.~\ref{prop:paraReduceSmall}:
\begin{enumerate}
\item Break $w$ into $w_K$, a product of boundedly many diagonal
  words, and $w_N$, a product of boundedly many off-diagonal blocks. (Cor.~\ref{cor:splitDiag})
\item Break $w_K$ into one diagonal word for each $S_i$. (Cor.~\ref{cor:splitWords})
\item Fill $w_N$ using Lemma~\ref{lem:infPres}.  (Lem.~\ref{lem:shortNP})
\end{enumerate}

First, though, we rid ourselves of the diagonal matrices in $w$.
Lemma~\ref{lem:infPres} lets us move diagonal matrices past shortcuts,
so we can shift the diagonal matrices to the beginning of $w$ using
$O(\ellw(w)^2)$ applications of relations.  Since each diagonal word
and off-diagonal block represents an element of $P^+$, the product of
the diagonal matrices is a diagonal matrix which is an element of
$P^+$.  Replace the three diagonal matrices with the product of $k$
diagonal matrices, one in each $\SL(S_i;\Z)$ (we think of each of these as a
diagonal word with one letter).  The resulting word,
which we call $w'$, is the product of up to $3p^2+p$ diagonal words
and up to $3p^2$ off-diagonal blocks.

The next step is to separate the diagonal words and the off-diagonal
blocks.  We need the following lemma:
\begin{lemma}\label{lem:collectDiags}
  Assume, as above, that $\#S_i\le p-2$ for all $i$. 
  If $v$ is a shortcut word in $\SL(S_a)$ which represents $M$ and
  $V\in \Z^{S_b}\otimes \Z^{S_c}$ for some $1\le b<c\le k$, then
  \begin{enumerate}
  \item \label{lem:collectDiags:ab} If $a=b$, then 
    $$\delta_\Gamma(v\short{n}_{bc}(V)v^{-1},\short{n}_{bc}(MV))=O((\ellw(v)+\log
    \|V\|_2)^2)$$
  \item \label{lem:collectDiags:ac} If $a=c$, then
    $$\delta_\Gamma(v\short{n}_{bc}(V)v^{-1},\short{n}_{bc}(VM^{-1}))=O((\ellw(v)+\log
    \|V\|_2)^2)$$
  \item \label{lem:collectDiags:distinct}
    If $a$, $b$, and $c$ are distinct, then 
    $$\delta_\Gamma([v,\short{n}_{bc}(V)])=O((\ellw(v)+\log
    \|V\|_2)^2)$$
  \end{enumerate}
\end{lemma}
\noindent \textit{Remark:} If we instead assume that $3\le \#S_i\le
p-3$ for all $i$, proving the lemma becomes much simpler.  For
example, if $3\le \#S_a\le p-3$, we can prove part
\ref{lem:collectDiags:distinct} by replacing the shortcuts in $v$ by
words in $\SL(S_a;\Z)$ and replacing the shortcuts in
$\short{n}_{bc}(V)$ by words in $\SL((S_a)^c;\Z)$.  Since the
corresponding sets of generators commute, we can commute the words at
quadratic cost.  When $\#S_a$ is particularly large or small, though,
we need to use more involved methods.

\begin{proof}
  We may assume that $\#S_a\ge 2$; otherwise, $v$ would be the empty
  word.  

  \noindent Parts \ref{lem:collectDiags:ab} and \ref{lem:collectDiags:ac}:

  We will mainly consider part \ref{lem:collectDiags:ab}; part
  \ref{lem:collectDiags:ac} is essentialy symmetric.

  Since $\short{n}_{bc}(V)$ is a product of shortcuts, we will
  show that
  $$\delta_\Gamma(v\short{e}_{ij}(x)v^{-1},\short{n}_{bc}(x Mz_i
  \otimes z_j))=O((\ellw(v)+\log |x|)^2)$$
  for every $x\in \Z$, $i\in S_b$, $j\in S_c$ and then apply that to
  each term of $\short{n}_{bc}(V)$.  

  Let $\omega_0=v\short{e}_{ij}(x)v^{-1}$.  First, we use
  Lemma~\ref{lem:shortEquiv} to replace the shortcuts in $v$ and
  $v^{-1}$ by words in $\SL(S)$ for some $S$.  If $\#S_b\ge 3$, we can
  take $S=S_b$, otherwise, take $l\in \{1,\dots,p\}$ such that
  $l\not\in S$, $l\ne j$, and let $S=S_b\cup \{l\}$.  Call the
  resulting word $v'$.  We can use the same lemma to replace
  $\short{e}_{ij}(x)$ by $\short{u}_S(x z_i \otimes z_j)$,
  transforming $\omega_0$ to
  $$\omega_1=v'\short{u}_S(x z_i \otimes z_j)(v')^{-1}.$$

  Finally, Lemma~\ref{lem:xiConj} applies to $\omega_1$, so we can
  transform it to $\short{u}_S(x M z_i \otimes z_j)$.  Since this is a
  curve in $H_{S,S^c}$, which has a quadratic Dehn function, we can
  use Theorem~\ref{thm:HDehn} and Lemma~\ref{lem:shortEquiv} to
  transform this to $\short{n}_{bc}(x M z_i \otimes z_j)$.  

  Applying this result to each term of $\short{n}_{bc}(V)$, we can
  transform $v\short{n}_{bc}(V)v^{-1}$ to
  $$\prod_{i,j} \short{n}_{bc}(v_{ij} Mz_i\otimes z_j)$$
  We can apply parts \ref{lem:infPres:add} and
  \ref{lem:infPres:commute} of Lemma~\ref{lem:infPres} to reduce this to 
  $\short{n}_{bc}(MV)$ as desired.  Part \ref{lem:collectDiags:ac}
  follows similarly.

  \noindent Part \ref{lem:collectDiags:distinct}:

  Since $\short{n}_{bc}(V)$ is a product of at most $p^2$ shortcuts,
  it suffices to show that if $v\in \SL(S)$ and $m,n\not\in S$, then
  $$\delta_\Gamma([v,\short{e}_{mn}(x)])=O((\ellw(v)+\log |x|)^2).$$
  If $2\le \#S\le p-3$, then we can use Lemma~\ref{lem:shortEquiv} to
  replace $v$ with a word in $\SL(S\cup \{m\})$ and prove the lemma by
  applying Lemma~\ref{lem:xiConj} to $H_{S\cup \{m\},(S\cup
    \{m\})^c}$.  It just remains to consider the case that $\#S=p-2$.

  Without loss of generality, we may take $S=\{2,\dots,p-1\}$.  Since
  $\#S\ge 3$, we can use Lemma~\ref{lem:shortEquiv} to replace $v$
  with a word $v'$ in $\SL(S)$.  We claim that 
  $$\delta_\Gamma([v',\short{e}_{1p}(x)])=O((\ellw(v)+\log |x|)^2).$$
  We will construct a homotopy from $v'\short{e}_{1p}(x)(v')^{-1}$ to
  $\short{e}_{1p}(x)$ through the curves
  \begin{align*}
    & v'[e_{12}(1),\short{e}_{2p}(x)](v')^{-1} & & \text{by
      Lemmas~\ref{lem:shortEquiv} and \ref{lem:infPres}} \\
    & [v'e_{12}(1){(v')}^{-1},v' \short{e}_{2p}(x){(v')}^{-1}] & & \text{by
      free insertion}\\
    &
    \biggl[\prod_{i=2}^{p-1}\short{e}_{1i}(m_i),\prod_{i=2}^{p-1}\short{e}_{ip}(n_i)\biggr]
    & & \text{by Lemma~\ref{lem:xiConj}}\\
    & \short{e}_{1p}(\sum_i m_i n_i)=\short{e}_{1p}(x) & &\text{by Lemma~\ref{lem:infPres}}
  \end{align*}
  Here, $m_i$ and $n_i$ are the coefficients of $Mz_2$ and $z_2M^{-1}$
  respectively.    The total cost of these steps is at most $O((\ellw(v)+\log
  |x|)^2)$.  The last step needs some explanation.  Let 
  $$w_1=\prod_{i=2}^{p-1}\short{e}_{1i}(m_i)$$
  $$w_2=\prod_{i=2}^{p-1}\short{e}_{ip}(n_i),$$
  so we are transforming $[w_1,w_2]$ to $\short{e}_{1p}(x)$.  Each
  term $\short{e}_{1i}(m_i)$ of $w_1$ commutes with every term of
  $w_1$ and $w_2$ except for $\short{e}_{ip}(n_i)$ and its inverse,
  and Lemma~\ref{lem:infPres} lets us transform
  $[\short{e}_{1i}(m_i),\short{e}_{ip}(n_i)]$ to
  $\short{e}_{1p}(m_in_i)$.  This commutes with every term of $w_1$
  and $w_2$.  So, if we use Lemma~\ref{lem:infPres} to move terms of
  $w_1$ past $w_2$, the only new terms that appear are of this form,
  so once we get rid of all of the terms of $w_1$ and $w_2$, we are
  left with 
  $$\prod_{i=2}^{p-1} \short{e}_{1p}(m_i n_i).$$
  Since $\sum_i m_in_i=z_2 M^{-1}\cdot M x z_2=x$, we can use
  Lemma~\ref{lem:infPres} to convert this to $\short{e}_{1p}(x)$.  All
  of the coefficients in this process are bounded by $\|M\|_2^2 x$,
  and $\|M\|_2$ is exponential in $\ellw(w)$, so this step
  has cost $O((\ell(w)+\log |x|)^2)$.  This concludes the proof.
\end{proof}

In particular, this lets us break $w$ into a product of diagonal
words and a product of off-diagonal blocks:
\begin{cor}\label{cor:splitDiag}
  If $g_1,g_2,g_3\in P$, $g_1g_2g_3=1$, and
  $$w=\omega(g_1)\omega(g_2)\omega(g_3),$$
  then there are words $w_K$ and $w_N$ such that $w_K$ is a product of
  at most $3p^2+p$ diagonal words, $w_N$ is a
  product of at most $3p^2$ off-diagonal blocks,
  $\ellw(w_K)=O(\ellw(w))$, $\ellw(w_N)=O(\ellw(w))$, and 
  $$\delta_\Gamma(w,w_Kw_N)=O(\ellw(w)^2).$$
\end{cor}
\begin{proof}
  As we noted before Lemma~\ref{lem:collectDiags}, it takes
  $O(\ellw(w)^2)$ applications of relations to replace $w$ by a word
  $w'$ which is a product of at most $3p^2+p$ diagonal words and at
  most $3p^2$ off-diagonal blocks.  Lemma~\ref{lem:collectDiags} lets
  us move diagonal words past off-diagonal blocks.  This process will
  affect the coefficients of these off-diagonal blocks, but it is
  straightforward to check that these coefficients remain bounded by
  $e^{\ellw(w)}$ throughout the entire process.  Thus, moving a
  diagonal word past an off-diagonal block always has cost
  $O(\ellw(w)^2)$.  We start with a bounded number of diagonal words
  and off-diagonal blocks, and no additional terms are created in the
  process, so we use Lemma~\ref{lem:collectDiags} only boundedly many
  times and the total cost remains $O(\ellw(w)^2)$.  The resulting
  word can be broken into a product of $3p^2+p$ diagonal words, which
  we call $w_K$, and a product of at most $3p^2$ off-diagonal blocks,
  which we call $w_N$.
\end{proof}

Next, we sort the diagonal words so that all the shortcut words in
$\SL(S_i;\Z)$ are grouped together for $i=1,\dots,k$.  We use the
following lemma:
\begin{lemma}\label{lem:shortCommute}
  Let $S, T\subset \{1,\dots, p\}$ be disjoint subsets such that
  $\#S,\#T\le p-2$.  Let $w_S$ be a shortcut word $\SL(S;\Z)$ and let
  $w_T$ be a shortcut word in $\SL(T;\Z)$.  Then
  $$\delta_\Gamma([w_S, w_T])=O((\ellw(w_S)+\ellw(w_T))^2).$$
\end{lemma}
\begin{proof}
  If $\#S\ge 3$ and $\#T\ge 3$, we can use Lemma~\ref{lem:shortEquiv}
  to replace $w_S$ and $w_T$ by words in $\SL(S;\Z)$ and $\SL(T;\Z)$,
  then commute the resulting words letter by letter.  Similarly, if
  $\#S=1$ or $\#T=1$, then $w_S$ or $w_T$ is trivial.  It remains only
  to study the case that one of $\#S$ and $\#T$ is 2.  Without loss of
  generality, we take $S=\{1,2\}$ and $T=S^c$.
  Consider the case that $w_T$ is a word in $\SL(S^c;\Z)$.

  Let $v=w_T$ be a word in $\SL(S^c;\Z)$ and consider
  $\delta_\Gamma([v,w_S])$.  Since $w_S$ is a shortcut word, we can
  write it as a product $w_S=w_1\dots w_n$ of diagonal matrices and
  shortcuts.  If $w_i$ is a diagonal matrix, it commutes with each
  letter of $v$; otherwise, we can bound the filling area of $[v,w_i]$ using
  part \ref{lem:collectDiags:distinct} of Lemma~\ref{lem:collectDiags},
  which states that
  $$\delta_\Gamma([w_i,v])\le c (\ellw(w_i)+\ellw(v))^2,$$
  so
  $$\delta_\Gamma([w_S,v])\le \sum_i \delta_\Gamma([w_i,v])\le c n (\ellw(w_S)+\ellw(v))^2.$$
  To get rid of the extra $n$, we need a slightly better bound.
  
  When $\ellw(v)\ge \ellw(w_i)$, we can get a stronger bound
  on $\delta_\Gamma([w_i,v])$ by breaking $v$ into segments of length
  $\sim \ellw(w_i)$.  Let $v=v_1\dots v_d$, where $\ellw(w_i)\le
  \ellw(v_j)\le 2\ellw(w_i)$ for each $j$.  Then
  $d\le \ellw(v)/\ellw(w_i)$ and 
  $$\delta_\Gamma([w_i,v_j])\le 9 c \ellw(w_i)^2,$$
  so 
  $$\delta_\Gamma([w_i,v])\le 9 c d \ellw(w_i)^2\le 9 c
  \ellw(v)\ellw(w_i).$$
  On the other hand, if $\ellw(v)< \ellw(w_i)$, then 
  $$\delta_\Gamma([w_i,v])\le 4 c \ellw(w_i)^2.$$
  So we have
  \begin{align*}
    \delta_\Gamma([w_S,v])&\le \sum_i \delta_\Gamma([w_i,v]) \\
    &    \le \sum_i 9 c  \ellw(v)\ellw(w_i)+4 c \ellw(w_i)^2 \\
    &    \le 9 c  \ellw(v)\ellw(w_S)+4 c \ellw(w_S)^2=O((\ellw(w_S)+\ellw(v))^2).
  \end{align*}

  So if $w_T$ is a word in $\SL(S^c;\Z)$, the lemma holds.  Otherwise,
  $w_T$ is a shortcut word in $\SL(S^c;\Z)$, and since $\#S^c\ge 3$,
  we can use Lemma~\ref{lem:shortEquiv} to replace $w_T$ with a word
  in $\SL(S^c;\Z)$ at cost $O(\ellw(w_T)^2).$ The lemma
  follows.
\end{proof}

We use this lemma repeatedly to sort the shortcut words in $w_K$.
\begin{cor}\label{cor:splitWords}
  If $v$ is a word representing the identity which is the product of
  at most $c$ diagonal words, then we can break $v$ into diagonal words
  $v_1,\dots, v_k$ at cost $O(c^2\ell(v)^2)$, where $v_i$ is a shortcut word in
  $\SL(S_i;\Z)$, $\ell(v_1\dots v_k)=\ell(v)$.
\end{cor}
\begin{proof}
  We just need to swap diagonal words in $v$ until all the diagonal
  words in $\SL(S_1;\Z)$ are at the beginning, followed by all the diagonal
  words in $\SL(S_2;\Z)$ and so on.  This takes at most $c^2$ swaps,
  and each swap has cost $O(\ellw(v)^2).$  Since $v$ represents the
  identity, each $v_k$ also represents the identity.
\end{proof}

So we can break the original $w$ into the $v_1,\dots, v_k$ and $w_N$
at cost $O(\ell(w)^2)$, and the $v_i$ each have
$\ell(v_i)=O(\ell(w))$.  To prove the lemma, it just remains to fill
$w_N$.  Recall that $w_N$ is a product of at most $3p^2$ off-diagonal
blocks.
\begin{lemma} \label{lem:shortNP} If $w=w_1\dots w_d$ is a product of
  $d$ off-diagonal blocks which represents the identity, then
  $$\delta_\Gamma(w)=O(d\ellw(w)^2).$$
\end{lemma}
\begin{proof}
  Let $g_i$ be the group element represented by $w_i$.  We will define
  a normal form $\omega_N$ for $N$ and then fill $w$ by bounding
  $$\delta_\Gamma(\omega_N(g_1\dots
  g_{i-1})\omega_N(g_i),\omega_N(g_1\dots g_{i})).$$
  We can combine these fillings into a filling for $w$.
  
  If $m\in N$, we can write $m$ in block upper-triangular form:
  $$m=\begin{pmatrix}
    I  & V_{12} & \dots  & V_{1k} \\
    0     & I     & \dots  & V_{2k} \\
    \vdots & \vdots  & \ddots & \vdots  \\
    0     & 0       & \dots  & I
  \end{pmatrix},$$
  where $V_{ij}\in \Z^{S_i}\otimes \Z^{S_j}$.
  We can decompose this into a product
  $$\omega_N(m)=x_k(m)\dots x_1(m)$$
  where
  $$x_i(m)=\short{n}_{i,i+1}(V_{i,i+1})\dots \short{n}_{i,k}(V_{i,k}).$$
  This is a normal form for elements of $N$; note that it has one term
  for each block above the diagonal.  Furthermore, each subword of $w$
  represents a single off-diagonal block, so there are $a_i$, $b_i$,
  and $V_i$ such that for all $i$,
  $\omega_N(g_i)=\short{n}_{a_i,b_i}(V_{i})$.

  So consider $\omega_N(g_1\dots g_{i-1})\short{n}_{a_i,b_i}(V_{i})$.
  The coefficients of $g_1\dots g_{i-1}$ are all bounded by
  $\exp(\ellw(w))$, so $\omega_N(g_1\dots g_{i-1})$ and
  $n=\short{n}_{a_i,b_i}(V_{i})$ both have length $O(\ellw(w))$.  We
  will transform $\omega_N(g_1\dots g_{i-1})n$ to $\omega_N(g_1\dots
  g_{i})$ by moving $n$ to the left until we can combine it with the
  right term in $\omega_N(g_1\dots g_{i-1})$.  We move $w$ to the left
  by commuting it with other off-diagonal blocks using
  Lemma~\ref{lem:infPres}.  That is, we make replacements:
  $$\short{n}_{ab}(V) n \to \begin{cases}
    n \short{n}_{ab}(V) & \text{if $a\ne b_i$ and $b\ne a_i$} \\
    n \short{n}_{ab}(V) \short{n}_{a,b_i}(VV_{i}) & \text{if $b=a_i$}.\\
    n \short{n}_{ab}(V) \short{n}_{a_i,b}(-V_{i}V) & \text{if $a=b_i$} 
  \end{cases}$$ One can use Lemma~\ref{lem:infPres} to make each of
  these replacements at cost
  $O(\ellw(w)^2)$.  Each replacement moves
  $n$ one term to the left and possibly creates one extra off-diagonal
  block.  We make replacements until $n$ is next to the
  $\short{n}_{a_i,b_i}(V')$-term in $\omega_N(g_1\dots g_{i-1})$, then
  use Lemma~\ref{lem:infPres} to combine the two terms into a single
  off-diagonal block.  Note that since we stop at the
  $\short{n}_{a_i,b_i}(V')$-term, we only move $n$ past
  $\short{n}_{ab}(V)$ when $a\le a_i$, so we never need the third
  case.

  During this process, we have inserted at most $k$ additional
  off-diagonal blocks.  Using a similar replacement process, we can
  move these new blocks to their places in $\omega_N(g_1\dots
  g_{i-1})$.  One can check that this does not add any further new
  blocks and has cost $O(\ellw(w)^2)$.  Thus
  $$\delta_\Gamma(\omega_N(g_1\dots g_{i-1})n,\omega_N(g_1\dots
  g_{i}))=O(\ellw(w)^2).$$
  Therefore,
  $$\delta_\Gamma(w)\le \sum_{i=1}^d \delta_\Gamma(\omega_N(g_1\dots
  g_{i-1})n,\omega_N(g_1\dots g_{i}))=O(d \ellw(w)^2),$$
  as desired.
\end{proof}

\subsection{Case 2: Large $S_i$'s}\label{subsec:specialCase}
The goal of this section is to prove:
\begin{prop}\label{prop:paraReduceLarge}  
  Let $P=U(p-1,1)$.  If $g_1,g_2,g_3\in P$ and $g_1g_2g_3=1$, let
  $$w=\omega(g_1)\omega(g_2)\omega(g_3).$$
  Then we can break $w$ into words $v_1,\dots, v_d$ at cost
  $O(\ell(w)^2)$, where $v_i$ is a shortcut word in $\SL(p-1;\Z)$ and
  $\sum \ell(v_i)^2=O(\ell(w)^2)$.
\end{prop}

Because some of our lemmas give poor bounds in this case, the proof of
Prop.~\ref{prop:paraReduceLarge} is very different from the proof of
Prop.~\ref{prop:paraReduceSmall}.  Instead of using combinatorial
manipulations to construct a filling, we use a variation of the
adaptive triangulation argument used to prove
Lemma~\ref{lem:redPara} to reduce the problem of filling $w$ to the
problem of filling $\omega$-triangles in parabolic subgroups of $P$.
We can then use Prop.~\ref{prop:paraReduceSmall} to fill such triangles.

Since $P$ is a finite-index extension of $\SL(p-1;\Z)\ltimes
\Z^{p-1},$ any word in $P$ which represents the identity can be
reduced to a word in $\SL(p-1;\Z)\ltimes \Z^{p-1}$ at cost linear in
the length of the word.  Dru{\c{t}}u showed that if $p\ge 4$, then the
group $H=\SL(p-1;\R)\ltimes \R^{p-1}$ has a quadratic Dehn function
\cite{DrutuFilling}, but we will need the stronger result that a curve
of length $\ell$ can be filled by a Lipschitz map of a disc of radius $\ell$. 

Let $\cE_H:=H/\SO(p-1)$.  The map $H\to \SL(p-1)$ induces a fibration
of $\cE_H$ over $\cE_{p-1}:=\SL(p-1)/\SO(p-1)$ with fiber $\R^{p-1}$.
Let $m:\cE_H\to \cE_{p-1}$ be this projection map.  If $x\in H$, let
$[x]_{\cE_H}$ be the corresponding point of $\cE_H$.  We will show:
\begin{lemma}\label{lem:drutuLipschitz}
  If $p\ge 4$, there is a $c_0$ such that for any $\gamma:[0,\ell]\to
  \cE_H$ which is a constant-speed parameterization of a closed curve of
  length $\ell$, if $\ell\ge 1$, and if $D^2(\ell):=[0,\ell]\times
  [0,\ell]$, then there is a $c_0$-Lipschitz map $f:D^2(\ell)\to \cE_H$
  which agrees with $\gamma$ on the boundary of $D^2(\ell)$.
\end{lemma}
The proof of this lemma requires some involved geometric and
combinatorial constructions, so we postpone it until the end of this
section.

Assuming the lemma, we can prove Prop.~\ref{prop:paraReduceLarge}.
First, we need to translate the notions in
Section~\ref{sec:depthFunction} to the context of $H$.  Let
$\cM_{p-1}:=\SL(p-1;\Z)\backslash \cE_{p-1}$.  Let $A^+,N^+\subset
\SL(p-1)$ be as in \ref{sec:depthFunction}, so that $N^+A^+$ is a
Siegel set in $\SL(p-1)$ and $\cS_{p-1}:=[N^+A^+]_{\cE_{p-1}}$ is a
fundamental set.  Let $N_H^+\subset H$ be the subset of $\R^{p-1}\in
H$ consisting of vectors with components in $[-1/2,1/2]$, and define
$\cS_H=[N_H^+N^+A^+]_{\cEH}$.  This is a fundamental set for the
action of $P$ on $\cEH$.  Just as we defined the map $\rho:\cE\to
\SL(p;\Z)$ using $\cS$, we can define a map $\rho_H:\cE_H\to P$ such
that $\rho_H(\cS_H)=I$ and $x\in \rho_{H}(x)\cS_H$ for all $x\in
\cE_H$.  Note that, unlike the previous case, $\cS_H$ is not Hausdorff
equivalent to $A^+$.

Define a depth function
$r:\cE_{p-1}\to \R^+$ by letting
$$r(x)=d_{\cM_{p-1}}([x]_{\cM_{p-1}},[I]_{\cM_{p-1}}).$$
Since $\cE_H$ fibers over $\cE_{p-1}$, we can define a depth function
$r_H:\cE_H\to \R^+$ by letting $r_H(x)=r(m(x))$.  

A lemma similar to Lemma~\ref{lem:depthLength} holds:
\begin{lemma}\label{lem:depthLengthH}
  There is a $c$ such that if $x,y\in \cE_H$, then 
  $$d_\Gamma(\rho_H(x),\rho_H(y))\le c(d_\cEH(x,y)+r_H(x)+r_H(y))+c$$
\end{lemma}
\begin{proof}
  If $x\in \cEH$ and $n_1\in N_H^+$, $n_2\in N^+$, and
  $a\in A^+$ are such that $x=[\rho_H(x) n_1n_2a]_\cEH$, then
  $r_H(x)=r([n_2a]_{\cE_{p-1}})$, so 
  $$d_\cEH([\rho_H(x)]_\cEH,x)\le d_H(I,n_1)+r_H(x)=r_H(x)+O(1).$$
  In particular,
  $$d_\cEH([\rho_H(x)]_\cEH,[\rho_H(y)]_\cEH)\le 2 \diam(N_H^+)+r_H(x)+r_H(y)+d_\cEH(x,y),$$
  and since $p\ge 5$, the lemma follows from Thm.~\ref{thm:LMR}.
\end{proof}

In addition, balls which are deep in the cusp are contained in a
translate of a parabolic subgroup.  Specifically,
Cor.~\ref{cor:parabolicNbhds} implies that there is a $c_1>0$ such
that if $x\in \cE_{H}$, $r_H(x)>c_1$, and $B$ is the ball of radius
$\frac{r_H(x)}{4(p-1)^2}$ around $x$, then $\rho_H(B)\subset g
U(q,p-1-q,1)$ for some $1\le q\le p-2$ and some $g\in P$.

We can thus use Cor.~\ref{cor:adaptive} to prove the following (cf.\
Lemma~\ref{lem:templateExist}):
\begin{lem}\label{lem:PTemplateExist}
  Let $p\ge 4$ and let $P=U(p-1,1)$.  There is a $c_2$ such that if
  $w$ is a word in $P$ which represents the identity and $\ell=\ellw(w)$,
  then there is a template $\tau$ for $w$ such that
  \begin{enumerate}
  \item  If $g_1, g_2, g_3\in P$
    are the labels of a triangle in the template, then either
    $\diam\{g_1,g_2,g_3\}\le c_2$ or there is a $q$ such that all of
    the $g_i$ are contained in the same coset of $U(q,p-1-q,1)$.
  \item If $g_1,g_2$ are the labels of an edge $e$ in the template,
    then 
    $$d_\Gamma(g_1,g_2)=O(\ell(e)).$$
  \item $\tau$ has $O(\ell^2)$ triangles, and if the $i$th triangle of
    $\tau$ has vertices labeled $(g_{i1},g_{i2},g_{i3})$, then
    $$\sum_{i}(d_\Gamma(g_{i1},g_{i2})+d_\Gamma(g_{i1},g_{i3})+d_\Gamma(g_{i2},g_{i3}))^2=O(\ell^2).$$

    Similarly, if the $i$th edge of $\tau$ has vertices labeled
    $h_{i1}, h_{i2}$, then 
    $$\sum_{i}d_\Gamma(h_{i1},h_{i2})^2=O(\ell^2).$$
  \end{enumerate}
\end{lem}
\begin{proof}
  The proof of the lemma proceeds in much the same way as the proof
  of Lemma~\ref{lem:templateExist}.  If $t$ is the smallest power of
  $2$ which is greater than $\ell$, we can use
  Lemma~\ref{lem:drutuLipschitz} to construct a Lipschitz map
  $f:[0,t]\times [0,t]\to \cE_H$ which takes the boundary of
  $D:=[0,t]\times [0,t]$ to the curve corresponding to $w$.  We assume
  that this is parameterized so that if $x$ is a lattice point in
  $\partial D$, then $f(x)=[g]_{\cE_H}$ for some $g\in P$.  This map
  has a Lipschitz constant bounded independent of $\ell$, say by
  $c_3$.

  Let $h:D\to \R$,
  $$h(x)=\max\{1,\frac{r_H(f(x))}{8(p-1)^2c_3}\}.$$
  This function is $1$-Lipschitz and we let $\tau_h$ be the
  triangulation of $D$ constructed in Cor.~\ref{cor:adaptive}.
  If $v$ is a vertex of $\tau_h$, we label it $\rho_H(f(v))$; this
  makes $\tau_h$ a template.  

  If $x$ is a lattice point on the boundary of $D$, then $f(x)\in
  [P]_{\cE_H}$ and so $h(x)=1$.  In particular, each lattice point on
  the boundary of $D$ is a vertex of $\tau_h$, so the boundary of
  $\tau_h$ is a $4t$-gon whose vertex labels fellow-travel with $w$.
  We can add $O(t)$ small triangles to $\tau_h$ to get a template
  $\tau$ for $w$.

  To prove that this satisfies the conditions of the lemma, we first need to
  show that if $v_1$ and $v_2$ are the endpoints of an edge of $\tau$,
  labeled by $g_1$ and $g_2$, then $d_{\Gamma}(g_1,g_2)\lesssim
  d(v_1,v_2)$.  By Lemma~\ref{lem:depthLengthH}, we know that 
  $$d_{\Gamma}(g_1,g_2)=O\left(d_{\cE_H}(f(v_1),f(v_2))+r_H(f(v_1))+r_H(f(v_2))+1\right),$$
  and each of these terms is $O(d(v_1,v_2))$, so
  $d_{\Gamma}(g_1,g_2)=O(d(v_1,v_2))$ as desired.

  Part 2 of the lemma then follows from the bounds in
  Cor.~\ref{cor:adaptive}.  Part 1 of the lemma follows from the
  fact that if $x_1$, $x_2$, and $x_3$ are the vertices of a
  triangle of $\tau$, with labels $g_1$, $g_2$, and $g_3$, then
  Cor.~\ref{cor:adaptive} implies that 
  $$\diam\{x_1,x_2,x_3\}\le 2  h(x_1)\le \frac{r_H(f(x))}{4(p-1)^2c_3}.$$
  If $h(x_1)$ is sufficiently large, then
  the remark before the lemma implies that $g_1,g_2$, and $g_3$ are in
  the same coset of $U(q,p-1-q,1)$ for some $j$.
\end{proof}

Proposition~\ref{prop:paraReduceLarge} follows as a corollary:
\begin{proof}[Proof of Prop.~\ref{prop:paraReduceLarge}]
  If $w$ is a word in $P$, let $\tau$ be a template for $w$ satisfying
  the properties in Lemma~\ref{lem:PTemplateExist}.  We can use $\tau$
  to break $w$ into a set of $\omega$-triangles and bigons.  Each of
  these either has bounded length or is a shortcut word in some
  $U(q,p-1-q,1)$.  The ones with bounded length can be filled with
  bounded area.  Since there are at most $O(\ell(w)^2)$ of these, this
  has total cost $O(\ell(w)^2)$.  The shortcut words can be broken
  into smaller pieces by using Prop.~\ref{prop:paraReduceSmall}.  This
  also has total cost $O(\ell(w)^2)$, and the resulting shortcut words
  $v_1,\dots,v_d$ fulfill the conditions of the proposition.
\end{proof}

So, to prove Proposition~\ref{prop:paraReduceLarge}, it suffices to
prove Lemma~\ref{lem:drutuLipschitz}.

\subsection{Proof of Lemma~\ref{lem:drutuLipschitz}}

In this section, we will show that Lipschitz curves in $\cE_H$ can be
filled with Lipschitz discs.  We will proceed by decomposing a loop in
$\cE_H$ into simple pieces.  First, we will recall an argument of
Gromov that in order to fill a Lipschitz loop with a Lipschitz disc,
it suffices to be able to fill a family of Lipschitz triangles with
Lipschitz discs (similar to the arguments using templates in
Section~\ref{sec:sketchProof}).  Since $H$ is a semidirect product,
these triangles can be chosen so that each side is a concatenation of
a geodesic in $\cE_{p-1}:=\SL(p-1;\R)$ and a curve that represents an
element of $\R^{p-1}$.  We complete the proof by filling polygons
whose sides consist of such curves with Lipschitz discs.

In this section, let $D^2(\ell):=[0,\ell]\times [0,\ell]$ and
$D^2:=D^2(1)$, and let $S^1(\ell)$ and $S^1$ be the boundaries of
$D^2(\ell)$ and $D^2$.

The following remarks and lemma about Lipschitz maps between polygons will be
helpful.
\begin{rem} \label{rem:polys}
  If $S$ is a convex polygon with non-empty interior and diameter at most $1$,
  there is a map $S\to D^2$ whose Lipschitz constant varies continuously
  with the vertices of $S$.
\end{rem}

\begin{rem} \label{rem:boundedFill}
  For every $\ell$, there is a $c(\ell)$ such that any closed
  curve in $\cE_H$ of length $\ell$ can be filled by a
  $c(\ell)$-Lipschitz map $D^2(1)\to \cE_H$.  This follows from
  compactness and the homogeneity of $\cE_H$.
\end{rem}

\begin{lem} \label{lem:reparam} Let
  $\gamma:S^1(\ell)\to X$ be a closed curve and let
  $\beta_0:D^2(\ell)\to X$ be a map such that $\beta_0|_{S^1(\ell)}$
  is some reparameterization of $\gamma$.  Then there is a map
  $\beta:D^2(\ell)\to X$ such that $\gamma=\beta|_{S^1(\ell)}$ such that
  $$\Lip(\beta)=O(\max\{\Lip(\beta_0),\Lip(\gamma)\}).$$
\end{lem}
\begin{proof}
  Let $\gamma_0=\beta_0|_{S^1(\ell)}$ and let
  $h:S^1(\ell)\times[0,\ell]\to X$ be a homotopy from $\gamma$ to
  $\gamma_0$.  Since the two curves differ by a reparameterization, we
  can choose $h$ to have Lipschitz constant of order $O(\max\{\Lip
  \gamma_0,\Lip \gamma\})$.  If we glue $\beta_0$ and $h$ together, we
  get a map $\beta_1:D'\to X$, where
  $$D'=\bigl[D^2(\ell) \cup (S^1(\ell)\times[0,\ell])\bigr]/\sim$$
  is the space obtained by identifying $S^1(\ell)\times \{1\}$ and $\partial
  D^2(\ell)$.  This map is a filling of $\gamma$ with Lipschitz constant
  $O(\max\{\Lip \gamma,\Lip \beta_0\})$, and since $D'$ is
  bilipschitz equivalent to $D^2(\ell)$, there is a map $\beta:D^2(\ell)\to X$
  which agrees with $\gamma$ on its boundary and has Lipschitz
  constant $O(\max\{\Lip \gamma,\Lip \beta_0\})$.
\end{proof}

One application of this remark is to convert homotopies to discs; if
$f_1,f_2:[0,\ell]\to X$ are two maps with the same endpoints, and
$h:[0,\ell]\times [0,\ell]\to X$ is a Lipschitz homotopy between $f_1$
and $f_2$ with endpoints fixed, then there is a disc filling
$f_1f_2^{-1}$ with Lipschitz constant $O(\Lip(h))$.

Suppose that $X$ is a metric space and $\omega$ is a normal form for
$X$.  That is, suppose that if $x,y\in X$, then $\omega(x,y):[0,1]\to X$ is a
constant-speed curve connecting $x$ and $y$ and that there is a $c$ such that
$\ell(\omega(x,y))\le c d(x,y)+c$.  (Note that we don't require
$\omega$ to satisfy any fellow-traveler properties.)  We may assume
that $\omega(x,x)$ is a constant curve for each $x$.  Then, just as we
built fillings of curves out of fillings of triangles in
Section~\ref{sec:sketchProof}, we can build Lipschitz discs by gluing
together Lipschitz triangles (cf.\ \cite{GroCC}).  Let $\Delta$ be the
equilateral triangle with side length 1.
\begin{prop}
  Let $X$ be a homogeneous riemannian manifold or a simplicial complex
  with bounded complexity, let $\omega$ be a normal form for $X$.
  Suppose that there is a $c_1$ such that for all $x,y,z\in X$, there
  is a map $f_{x,y,z}:\Delta\to X$ which takes the sides of the
  triangle to $\omega(x,y)$, $\omega(y,z)$, and $\omega(x,z)^{-1}$ and
  such that $\Lip f_{x,y,z}\le c_1 \diam\{x,y,z\}+c_1$.  Then there is
  a $C$ such that for every unit-speed Lipschitz closed curve
  $\gamma:[0,\ell]\to X$ of length $\ell \ge 1$, there is a map
  $g:D^2(\ell)\to X$ which agrees with $\gamma$ on the boundary and
  has $\Lip g\le C$.
\end{prop}
\begin{proof}
  \begin{figure}
    \def\svgwidth{4in}
    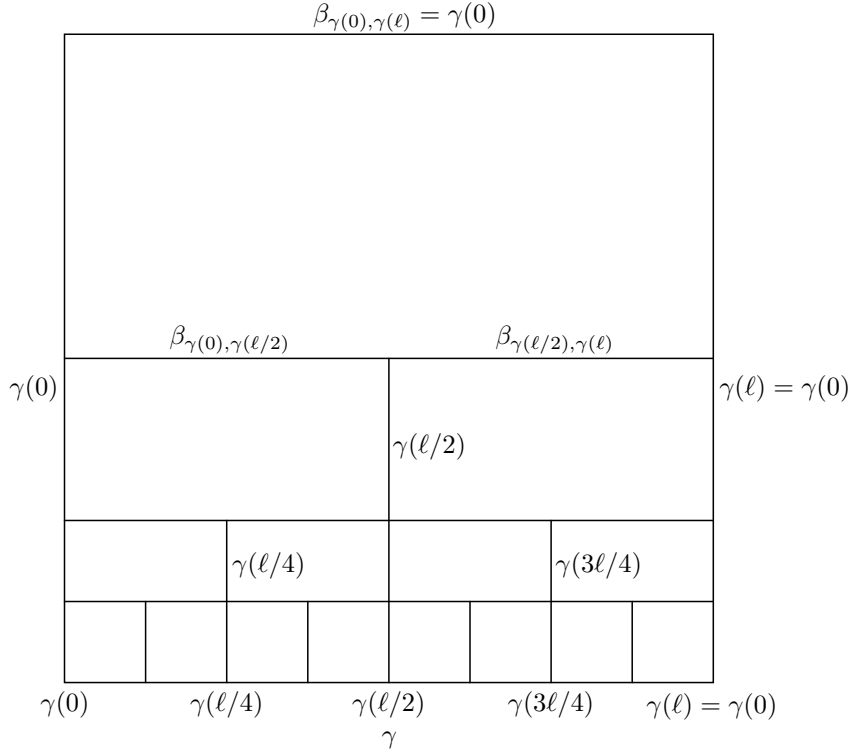
    \caption{\label{fig:paraLipDehn}A filling of $\gamma$ in $X$}
  \end{figure}
  
  We will construct a map $g:[0,\ell]\times[0,\ell]\to X$ which agrees
  with $\gamma$ on one side and stays constant on the other three
  sides.  The construction is essentially Gromov's construction of
  Lipschitz extensions from \cite{GroCC}.  Let $k=\lceil \log_2
  \ell \rceil$.  We will construct $g$ as in Figure~\ref{fig:paraLipDehn}.  The figure depicts a
  decomposition of $[0,\ell]\times [0,\ell]$ into $k+1$ rows of rectangles;
  the top row has one $\ell\times \frac{\ell}{2}$ rectangle, while the $i$th
  from the top consists of $2^{i-1}$ rectangles of dimensions
  $2^{-i+1}\ell\times 2^{-i}\ell$.  The bottom row is an exception, consisting
  of $2^k$ squares of side length $2^{-k}\ell$.  Call the resulting
  complex $D$.

  We label all the edges of $D$ by curves in $X$.  First, we label
  all the vertical edges by constant curves; the vertical
  edges with $x$-coordinate $a$ are labeled by $\gamma(a)$.  We label
  horizontal edges using the normal form:  the edge from $(x_1,y)$ to
  $(x_2,y)$ is labeled by $\omega(\gamma(x_1),\gamma(x_2))$, except for the
  bottom edge of $D$, which is labeled $\gamma$.  We can
  then define $g$ on the $1$-skeleton of $D$ by sending each edge to
  the constant-speed parameterization of its label.  It is easy to
  check that this construction is Lipschitz, with Lipschitz constant
  of order $O(\ell)$.

  Let $R$ be a rectangle in $D$.  If $R$ is in the bottom row of
  cells, then $g$ maps the boundary of $R$ to a curve of length
  bounded independently of $\gamma$, so we may extend $g$ over $R$
  using Remark~\ref{rem:boundedFill}, and all these extensions have
  Lipschitz constant bounded independently of $\gamma$.  Otherwise,
  suppose that $R$ is a $2^{-i+1}\ell\times 2^{-i}\ell$ rectangle.  Then,
  since $g$ maps both its vertical edges to points, the restriction of
  $g$ to the boundary of $R$ is a curve of the form
  $\omega(x,y)\omega(y,z)\omega(x,z)^{-1}$, where $x=\gamma(t)$,
  $y=\gamma(t+2^{-i}\ell)$, and $z=\gamma(t+2^{-i+1}\ell)$.  By assumption,
  there is a map $f_{x,y,z}:\Delta\to X$ which fills this curve and
  has Lipschitz constant $\le c_1 2^{-i-1}\ell +c_1$.  Since $R$ is
  bilipschitz equivalent to $D^2(2^{-i-1}\ell)$, we can reparameterize
  $f_{x,y,z}$ to get a map $R\to X$ which agrees with $g$ on $\partial
  R$ and has Lipschitz constant bounded independently of $\gamma$ and
  $i$.

  Defining extensions like this on every rectangle gives us a map
  $g:[0,\ell]\times [0,\ell]\to X$ whose boundary is a
  reparameterization of $\gamma$ and whose Lipschitz constant is
  bounded independent by some $C_0$ independent of $\gamma$, so the
  proposition follows by applying Lemma~\ref{lem:reparam}.
\end{proof}

Now we construct a normal form $\omega_{H}$ for $\cEH$.  First, for each
$h\in H$, we will construct a curve $\lambda_h:[0,1]\to H$ which
connects $I$ to $h$.  We can write $h$ as
$$h=\begin{pmatrix}
  M & v \\
  0 & 1
\end{pmatrix}$$ for some $M\in \SL(p-1)$ and $v\in \R^{p-1}$; we
denote the corresponding unipotent matrix in $H$ by $u(v)$.  Let
$\gamma_M$ be a geodesic in $\SL(p-1)$ which connects $I$ to $M$.  We
will construct $\lambda_h$ by concatenating $\gamma_M$ and a
curve $\psi(v):[0,1]\to H$ which connects $I$ to $u(v)$.

If $v=0$, let $\psi(v)$ be constant.  If $v\in \R^{p-1},$ $v\ne 0$, we
can write $v=\kappa \bar{v}$, where $\kappa:=\max\{ \|v\|_2,1\}$ and
$\bar{v}:=\frac{v}{\kappa},$ so that $\|\bar{v}\|_2\le 1$ and $0\le \log
\kappa =O(d_H(I,u(v)))$.  Let
$$v_1=\frac{v}{\|v\|_2}, v_2, \dots,v_{p-1}\in \R^{p-1}$$
be an orthonormal basis of $\R^{p-1}$.  Let $D(v)$ be the matrix which
stretches $v_1$ by a factor of $\kappa$ and shrinks the rest of the
$v_i$ by a factor of $\kappa^{1/(p-2)}$.  This is positive definite
and has determinant 1.  Furthermore, $D(v)\bar{v}=v$, so
$u(v)=D(v)u(\bar{v})D(v)^{-1}$.  Let $\D(v)$ be the curve $t\mapsto
D(v)^t$ for $0\le t\le 1$, let $\U(v)$ be the curve $t\mapsto u(tv)$
for $0\le t\le 1$, and define $\psi(v)$ to be the concatenation
$\psi(v)=\D(v)\U(\bar{v})\D(v)^{-1}$.  This has length $O(\blog
\|v\|)$, where $\blog x = \max\{1,\log x\}$.  Define
$\lambda_h=\gamma_M\psi(v)$.  It is easy to check that this satisfies
the desired length bounds.

If $x,y\in \cEH$, choose lifts $\tilde{x}, \tilde{y}$ such that
$[\tilde{x}]_\cEH=x$ and $[\tilde{y}]_\cEH=y$.  Since
$\SO(p-1)$ has bounded diameter,
different choices of lift only differ by a bounded distance.
Define $\omega_H(x,y)$ so that
$$\omega_H(x,y)(t)=[\tilde{x}\lambda_{\tilde{x}^{-1}\tilde{y}}(t)]_\cEH.$$
It is easy to check that this
satisfies the desired length bounds.

Next, we will construct discs filling triangles whose sides are in
normal form.  We claim:
\begin{lemma}\label{lem:LipCombFill}
  If $h_1, h_2\in H$ and
  $$w=\tilde\omega_H(h_1)\tilde\omega_H(h_2)\tilde\omega_H(h_1h_2)^{-1},$$
  then there is a filling $f:D^2\to \cE_H$ of $[w]_\cEH$ such that
  $\Lip(f)=O(d(I,h_1)+d(I,h_2))$.
\end{lemma}
To prove this lemma, we will follow the template of Figure~\ref{fig:LipCombFill}.
\begin{figure}
  \def\svgwidth{2.5in}
  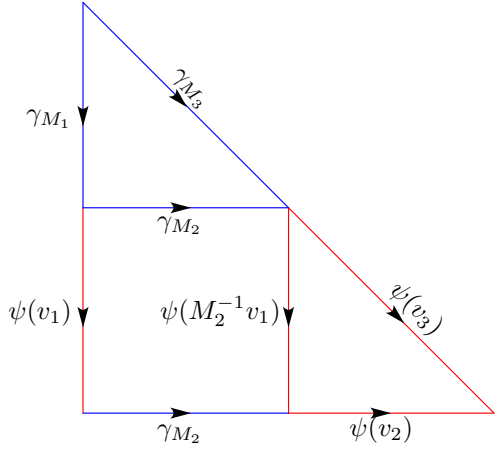
  \caption{\label{fig:LipCombFill}A quadratic filling of
    $\omega_H(M_1v_1)\omega_H(M_2v_2)\omega_H(M_3v_3)^{-1}$}
\end{figure}
The figure suggests that a filling for $w$ can be constructed from
fillings for two triangles and a rectangle.  The following lemmas will
construct those fillings.

\begin{lemma}\label{lem:lipPsiConj}
  Let $\gamma:[0,1]\to \SL(p-1)$ be a curve connecting $I$ and $M$ and
  let $v\in \R^{p-1}$.  There is a map $f:D^2\to \cEH$ which sends
  the boundary of $D^2$ to the curve
  $[\gamma \psi(v)\gamma^{-1} \psi(Mv)^{-1}]_\cEH$ and which has Lipschitz constant
  $\Lip f=O(\blog \|v_1\|_2+\ell(\gamma))$.
\end{lemma}

\begin{lemma}\label{lem:LipShortenedTriFill}
  Let $v_1,v_2\in \R^{p-1}$.  There is a map $f:D^2\to \cEH$ which sends
  the boundary of $D^2$ to the curve
  $[\psi(v_1)\psi(v_2)\psi(v_1+v_2)^{-1}]_\cEH$ which has Lipschitz constant
  $\Lip f=O(\blog \|v_1\|_2+\blog \|v_2\|_2)$.
\end{lemma}

If we assume these two lemmas, then the proof of
Lemma~\ref{lem:LipCombFill} follows easily:
\begin{proof}[Proof of Lemma~\ref{lem:LipCombFill}]
  Let $X$ be a triangle decomposed as in Figure~\ref{fig:LipCombFill}.
  We put a metric on $X$ so that the two small triangles are isoceles
  right triangles with legs of length 1 and the corner rectangle is a
  square with side length 1.  Under this metric, $X$ is bilipschitz
  equivalent to $D^2$.  Let $\ell=\ellc(w)$.  Construct a map on the
  $1$-skeleton of $X$ so that if an edge is labeled by a curve
  $\gamma$ in the figure, it is sent to a constant-speed
  parameterization of $[\gamma]_\cEH$.  It is easy to check that the
  Lipschitz constant of this map is of order $O(\ell)$.  We can use
  Lemmas~\ref{lem:LipShortenedTriFill} and \ref{lem:lipPsiConj} to
  construct maps from the lower right triangle and the corner square
  to $\cEH$ with Lipschitz constants of order $O(\ell)$.

  It only remains to construct a filling of the upper triangle.  The
  boundary of the upper triangle is a curve in
  $\cE_{p-1}=\SL(p-1)/\SO(p-1)$, so we can construct a Lipschitz
  filling of the upper triangle by coning it off by geodesics.  This
  filling also has Lipschitz constant of order $O(\ell)$, completing
  the construction.
\end{proof}

It remains to prove the two lemmas.
\begin{proof}[Proof of Lemma~\ref{lem:lipPsiConj}]
  Let $w=\gamma \psi(v)\gamma^{-1} \psi(Mv)^{-1}$.  If $v=0$, then
  $w=\gamma\gamma^{-1}$, so we may assume that $v\ne 0$.  If
  $\ellc(w)\le 1$, we can use Remark~\ref{rem:boundedFill} to fill
  $w$, so we also assume that $\ellc(w)\ge 1$.

  Recall that $\psi(v)$ is defined as the concatenation
  $\D(v)\U(\bar{v})\D(v)^{-1}$, where $\bar{v}$ is a vector parallel to $v$
  with length at most 1.  We can thus write
  $$w=\gamma \D(v)\U(\bar{v}) \D(v)^{-1}\gamma^{-1} \D(Mv)
  \U(-\overline{Mv})\D(Mv)^{-1}.$$ Let
  $$\gamma'=\D(Mv)^{-1} \gamma \D(v).$$
  Changing the basepoint of $w$, we obtain the curve
  $$w_1=\gamma' \U(\bar{v})(\gamma')^{-1}\U(-\overline{Mv}).$$
  This can be filled by a map of the form
  $$\beta(x,t)=\gamma'(x)u(t\cdot \gamma'(x)^{-1}\overline{Mv})$$
  (see Figure~\ref{fig:lipXiConj:Rect}).  This filling has a foliation
  (horizontal curves in the figure) consisting of curves 
  $\U(\gamma'(x)^{-1}\overline{Mv})$, but these may be exponentially large.
  We will use  a homotopy in $\SL(p-1)$ to replace $\gamma'$ by a
  curve $\sigma$ such that the length of $\sigma(x)^{-1}\overline{Mv}$ is
  always bounded.

  \begin{figure}
    \def\svgwidth{2in}
    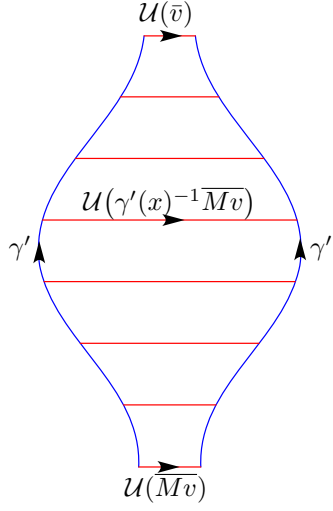
    \caption{\label{fig:lipXiConj:Rect}An exponential filling of $\gamma' \U(\bar{v})(\gamma')^{-1}\U(-\overline{Mv})$}
  \end{figure}

  First, we construct $\sigma$.  Let 
  $$S:=\{m\mid m\in \SL(p-1), \|m^{-1} \overline{Mv}\|_2\le 1\},$$
  and let
  $$M':=D(Mv)^{-1}MD(v)$$
  be the endpoint of $\gamma'$.  Since $\bar{v}=(M')^{-1}\overline{Mv}$, the
  endpoint of $\gamma'$ lies in $S$, and we will construct a
  curve $\sigma$ in $S$ which connects the identity to $M'$.

  Consider the case that $\overline{Mv}=\bar{v}$, so that $M'$ is in
  the stabilizer of $\overline{Mv}$, which we write  $\SL(p-1)_{\overline{Mv}}$.  This stabilizer is contained in
  $S$ and is isomorphic to $\SL(p-2)\ltimes \R^{p-2}$.  Furthermore,
  it is connected and since $p\ge 5$, its inclusion in $\SL(p-1)$ is
  undistorted, so we can let $\sigma$ be a path in
  $\SL(p-1)_{\overline{Mv}}$ between $I$ and $M'$.

  To construct $\sigma$ in the general case, it suffices to construct
  a curve in $S$ which connects $M'$ to a point in $\SL(p-1)_{\overline{Mv}}$;
  we can then apply the previous case.
  If $\|\overline{Mv}\|_2=\|\bar{v}\|_2$, we can take this to be a
  curve in
  $\SO(p-1)$ of bounded length.  Otherwise, we can construct a path of
  matrices that shrink (or grow) $\overline{Mv}$ and grow (or shrink)
  all the
  perpendicular directions; this path can be taken to lie in $S$ and
  its length is at most 
  $$O\left(\left|\log \|\overline{Mv}\|_2-\log \|\bar{v}\|_2\right|\right)\le  O(\log \|M'\|_2)$$

  We will use $\sigma$ to construct a map $f:[0,2\ellc(w)+1]\times
  [0,\ellc(w)]\to \cEH$ whose boundary is a parameterization of $w$.
  The domain of this map is divided into two $\ellc(w)\times \ellc(w)$
  squares and a $\ellc(w)\times 1$ rectangle
  (Fig.~\ref{fig:lipXiConj}); the squares will be homotopies in
  $\SL(p-1)$.  We will map the boundaries of each of these shapes into
  $\cE_H$ by Lipschitz maps, then construct Lipschitz discs in $\cE_H$
  with those boundaries.

  Let $f$ take the $1$-skeleton of the rectangle into $\cEH$ as
  labeled in the figure, parameterizing each edge with constant speed.
  The boundaries of the shapes in the figure are then
  $[\sigma\gamma^{-1}]_{\cEH}$ and $$w_2=[\sigma \U(\bar{v}) \sigma^{-1}
  \U(-\overline{Mv})]_{\cEH}.$$ The first curve, $\sigma\gamma^{-1}$, is a
  closed curve in $\SL(p-1)$ of length $O(\ellc(w))$.  Since
  $\SL(p-1)/\SO(p-1)$ is non-positively curved, the projection to
  $\cEH$ has a filling in $\cE_H$ with area $O(\ellc(w)^2)$.  This can
  be taken to be a $c$-Lipschitz map from $D^2(\ellc(w))$, where $c$
  depends only on $p$.

  The second curve is the
  boundary of a ``thin rectangle''.  That is, there is a Lipschitz map
  $$\rho:[0,\ellc(w)] \times [0,1]\to H$$
  $$\rho(x,t)=\sigma(x)u(t \sigma(x)^{-1} \overline{Mv})=u(t
  \overline{Mv})\sigma(x)$$
  which sends the four sides of the rectangle to $\sigma, \U(\bar{v}),
  \sigma^{-1},$ and $\U(-\overline{Mv})$.  Projecting this disc to $\cEH$
  gives a Lipschitz filling of $w_2.$ 

  We glue these discs together to get a Lipschitz map from the
  rectangle to $\cEH$.  The boundary of the rectangle is a Lipschitz
  reparameterization of $[w]_\cEH$, so we can use Lemma~\ref{lem:reparam} to
  get a filling of $[w]_\cEH$ by a disc $D^2(\ell(w))$ with Lipschitz constant of order $O(1)$.
  Rescaling this gives a filling of $[w]_\cEH$ by the disc $D^2$ with 
  Lipschitz constant of order $O(\ellc(w))$ as desired.

  \begin{figure}
    \def\svgwidth{4in}
    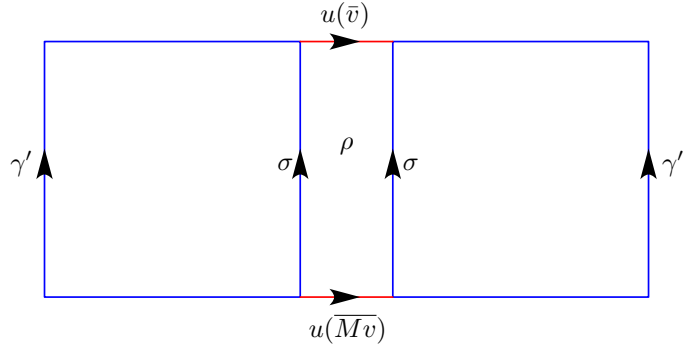
    \caption{\label{fig:lipXiConj}A quadratic filling of $\gamma' \U(\bar{v})(\gamma')^{-1}\U(-\overline{Mv})$}
  \end{figure}
\end{proof}

\begin{proof}[Proof of Lemma~\ref{lem:LipShortenedTriFill}]
  Let $w=\psi(v_1)\psi(v_2)\psi(v_1+v_2)^{-1}$.  As before, we may
  assume that $\ellc(w)>3$.  Let $S=\langle v_1, v_2\rangle\subset
  \R^{p-1}$ be the subspace generated by the $v_i$ and let
  $\lambda=\max\{\|v_1\|_2, \|v_2\|_2,\|v_1+v_2\|_2\}$.  Let $D\in
  \SL(p-1)$ be the matrix such that $Ds=\lambda s$ for $s\in S$ and
  $Dt=\lambda^{-1/(p-1-\dim(S))} t$ for vectors $t$ which are
  perpendicular to $S$; this is possible because $\dim(S)\le 2$ and
  $p\ge 5$.

  Let $\gamma_D$ be the curve $t\mapsto D^t$ for $0\le t\le 1$; this has
  length $O(\log \lambda)=O(\ellc(w))$.  We construct a filling of
  $[w]_\cEH$ based on a triangle with side length $\ellc(w)$ as in
  Figure~\ref{fig:shortTri}.  The central triangle in the figure has
  side length $1$; since $\ellc(w)\ge 3$, the trapezoids around the
  outside are bilipschitz equivalent to discs $D^2(\ell)$ with
  Lipschitz constant bounded independently of $w$.  Let $f$ take each
  edge to $H$ as labeled, and give each edge a constant-speed
  parameterization; $f$ is Lipschitz on each edge, with a Lipschitz
  constant independent of the $v_i$.  Let $\bar{f}$ be the projection
  of $f$ to $\cEH$.  We've defined $\bar{f}$ on the edges in the
  figure; it remains to extend it to the interior of each cell.
  
  The map $\bar{f}$ sends the boundary of the center triangle to a
  curve of length at most 3, so we can use Rem.~\ref{rem:boundedFill}
  to extend $\bar{f}$ to its interior.  The map $\bar{f}$ sends the
  boundary of each trapezoid to a curve of the form
  \begin{equation}[\psi(v_i)^{-1}\gamma_D \U(\lambda
    v_i)\gamma_D^{-1}]_\cEH.
  \end{equation}
  Lemma~\ref{lem:lipPsiConj} gives Lipschitz discs filling such curves.
  Each of these fillings has Lipschitz constant bounded independently
  of $w$, so the resulting map on the triangle is a filling of
  $[w]_\cEH$ by a triangle of side length $\ellc(w)$ with Lipschitz
  constant bounded independently of $w$.  By rescaling and mapping the
  triangle to $D^2$, we obtain a filling of $[w]_\cEH$ by the disc
  $D^2$ with 
  Lipschitz constant of order $O(\ellc(w))$ as desired.

  \begin{figure}
    \def\svgwidth{3in}
    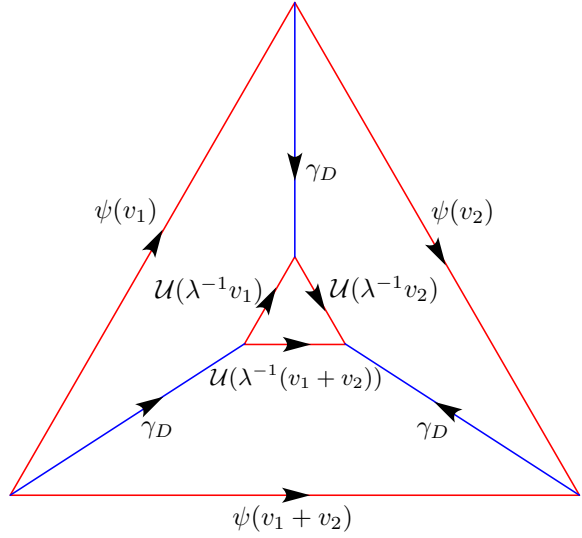
    \caption{\label{fig:shortTri}A quadratic filling of $\psi(v_1)\psi(v_2)\psi(v_1+v_2)^{-1}$}
  \end{figure}
\end{proof}

\section{The base case: $\SL(2;\Z)$}\label{sec:baseCaseProof}

In this section, we will prove Lemma~\ref{lem:baseCase},  which states
that if $w$ is a shortcut word in $\SL(2;\Z)$, then
$$\delta_\Gamma(w)=O(\ell^2).$$

The proof uses the adaptive template methods developed in
Sec.~\ref{sec:redParaProof}.  The main change from
Sec.~\ref{sec:redParaProof} is that the curve that we fill will not be
in the thick part of $\cE_2$.

  Let $w=w_1\dots w_n$ be a shortcut word representing the identity,
  where each $w_i$ is either a diagonal matrix in $\SL(2;\Z)$ or a
  shortcut of the form $\short{e}_{12}(x)$ or $\short{e}_{21}(x)$.  
  We first use Lemma~\ref{lem:infPres} to replace all occurrences of
  $\short{e}_{21}(x)$ in $w$ by $g \short{e}_{12}(-x)g^{-1}$, where $g$ is a
  word representing a Weyl group element.  This has cost
  $O(\ellw(w)^2)$, and it lets us assume that $\short{e}_{21}(x)$ does
  not occur in $w$ for $|x|\ge 1$.

  For each $i$, let $w(i)$ be the group element represented by
  $w_1\dots w_i$.  Let $\cS_2$ be a Siegel set for $\SL(2;\Z)$.
  For each $i$, we will construct a curve
  $\alpha_i:[0,\ellw(w_i)]\to \cE_2$ which connects $[w(i)]_{\cE_2}$ to
  $[w(i+1)]_{\cE_2}$ such that:
  \begin{itemize}
  \item The curves $\alpha_i$ are uniformly Lipschitz, with Lipschitz
    constants bounded independently of $w$.
  \item There is an integer $t_i\in [0,\ellw(w_i)]$ such that
    if $0\le j\le t_i$ is an integer, then $\alpha_i(j)\in w(i) \cS_2$ and 
    if $t_i<  j\le \ellw(w_i)$, then $\alpha_i(j)\in w(i+1) \cS_2$.
  \end{itemize}
  If $\ellw(w_i)< 3$, we define $\alpha_i$ on $[0,1]$ as the geodesic connecting
  $[w(i)]_{\cE_2}$ and $[w(i+1)]_{\cE_2}$ and we define $\alpha_i$ on
  $[1,\ell(w_i)]$ to be the constant value $[w(i+1)]_{\cE_2}$.  We
  let $t_i=0$.

  If $\ellw(w_i)\ge 3$, let $x$ be such that
  $w_i=\short{e}_{12}(x)$.  Let
  $$D=\begin{pmatrix}e & 0 \\
    0 & e^{-1}
  \end{pmatrix}$$
  and note that $[D^x]_{\cE_2}\in \cS_2$ for all $x\ge 0$.  Let
  $t_i=\left\lceil \frac{\ellw(w_i)}{3}\right\rceil$, and let
  $\beta:[0,\ellw(w_i)]\to\SL(2;\R)$ be the concatenation of
  geodesic segments connecting
  \begin{align*}
    p_1&=w(i)\\
    p_2&=w(i)D^{\log(|x|)/2}\\
    p_3&=w(i)D^{\log(|x|)/2}e_{12}(\pm 1)\\
    p_4&=w(i)D^{\log(|x|)/2}e_{12}(\pm 1)D^{-\log(|x|)/2}=w(i)e_{12}(x)=w(i+1).
  \end{align*}
  Here the sign of $\pm1$ is the same as the sign of $x$.
  Parameterize this curve so that $\beta|_{[0,t_1]}$ connects $p_1$
  and $p_2$, $g|_{[t_1,t_1+1]}$ connects $p_2$ and $p_3$, and
  $g|_{[t_1+1,\ellw(w_i)]}$ connects $p_3$ and $p_4$.  Let
  $\alpha_i=[\beta]_{\cE_2}$.  This curve has velocity bounded
  independently of $x$.  Furthermore, if $t\in \Z$, then
  $\alpha_i(t)\in w(i)\cS_2$ if $t\le t_1$ and $\alpha_i(t)\in w(i+1)
  \cS_2$ if $t> t_1$.

  Let $\alpha:[0,\ellw(w)]\to \cE_2$ be the concatenation of
  the $\alpha_i$.  From here, we largely follow the proof of
  Lemma~\ref{lem:templateExist}; we construct a filling $f$ of
  $\alpha$, an adaptive triangulation $\tau$, and a template based on
  $\tau$ so that a vertex $x$ of $\tau$ is labeled by an element
  $\gamma$ such that $f(x)\in \gamma \cS_2$.

  Let $d$ be the smallest power of $2$ larger than $\ellw(w)$ and let
  $\alpha':[0,d]\to \cE_2$ be the extension of $\alpha$ to $[0,d]$,
  where $\alpha'(t)=[I]_{\cE_2}$ when $t\ge \ellw(w)$.  Let
  $D^2(d)=[0,d]\times [0,d]$.  We can map $\partial D^2(d)$ into
  $\cE_2$ by sending one side to $\alpha'$ and sending the other three
  sides to $[I]_{\cE_2}$, and since $\cE_2$ is nonpositively curved,
  we can extend this map to all of $D^2(d)$ by coning it to a point
  along geodesics.  Call the resulting map $f:D^2(d)\to \cE_2$.  This
  is $c$-Lipschitz for some $c$ independent of $w$ and has area
  $O(\ellw(w)^2)$.

  Let $\cM_2=\SL(2;\Z)\backslash \cE_2$ and define the depth
  function $r:\cE_2\to \R+$ by
  $$r(x)=d_{\cM_2}([x]_{\cM_2},[I]_{\cM_2})$$
  Let $h:[0,d] \times [0,d]$ be 
  $$h(x)=\max\{1,\frac{r(f(x))}{32 c}\}$$
  This is a 1-Lipschitz function, so we can use
  Cor.~\ref{cor:adaptive} to construct a triangulation $\tau_h$ of
  $[0,d]\times [0,d]$.  It remains to convert this triangulation into
  a template.  
  
  For each vertex $v$ of $\tau_h$, we label $v$ by an element $g\in
  \SL(2;\Z)$ such that $f(v)\in g\cS_2$.  For the interior vertices,
  any such element suffices.  For the boundary vertices, we must make
  choices that agree with $w$.  Let $\ell_i=\sum_{j=1}^i \ellw(w_j)$
  so that $\alpha_i$ and $\alpha_{i+1}$ meet at $(\ell_i,0)$.  We have
  $\alpha'(\ell_i)=[w(i)]_{\cE_2}$, so $h(\ell_i,0)=1$ and
  $(\ell_1,0)$ must be a vertex of $\tau_h$; we label it $w(i)$.  Let
  $\beta_0=0$, $\beta_i=\ell_{i-1}+t_{i}$ for $i=1,\dots,n$, and
  $\beta_{n+1}=d$, so that if $\beta_i < t\le \beta_{i+1}$, then
  $f(t,0)=\alpha'(t)\in w(i)\cS_2$.  For all $t$ with $\beta_i < t\le
  \beta_{i+1}$, label the point $(t,0)$ by the element $w(i)$.
  Boundary vertices that are not of the form $(t,0)$ are all sent to
  $[I]_{\cE_2}$ under $f$, and we label them by $I$.  With this
  labeling, the boundary word of $\tau_h$ is $w$.

  A filling of the triangles in $\tau_h$ thus gives a filling of $w$.
  As in Lemma~\ref{lem:templateExist}, each triangle of $\tau_h$ either has
  short edges and thus a bounded filling area, or has vertices whose
  labels lie in a translate of a parabolic subgroup.  In this case,
  that parabolic subgroup must be $U(1,1)$, and
  Lemma~\ref{lem:infPres} allows us to fill any such triangle with
  quadratic area.  Cor.~\ref{cor:adaptive}.(3) thus implies that 
  $\delta(w)=O(\ellw(w)^2)$, as desired.

\section{Open questions}\label{sec:open}

One natural open question is whether these results can be extended to
a proof of Conj.~\ref{conj:mainConj}, or, as an important special
case, whether they can be used to find a bound on the Dehn function of
$\SL(4;\Z)$.  Some parts of the proof, especially the geometric lemmas
in Sec.~\ref{sec:redParaProof} extend naturally to other lattices in
semisimple groups.  That is, if $\Gamma$ acts on a symmetric space
$\cE$, one can define a fundamental set $\cS$ which is a union of
Siegel sets, use $\cS$ to define a map $\rho:\cE\to \Gamma$, and show
that if $x$ and $y$ are close together and deep in a cusp, then
$\rho(x)$ and $\rho(y)$ lie in a coset of some parabolic subgroup.
Using this fact, one can find various ways to construct templates
whose triangles have vertices lying in parabolic subgroups.

For $\SL(p;\Z)$, we filled these triangles using combinatorial lemmas,
but these lemmas are hard to generalize to other groups.  In general,
appropriate analogues of Lemma~\ref{lem:infPres} and
Lemma~\ref{lem:shortEquiv} should lead to a polynomial bound
on the Dehn function of a lattice.  One way to get such a bound is to
use a template consisting of simplices all of size $\sim 1$, as in
\cite{YoungQuart}.  In this case, each edge can be labeled by a group
element which lies in a parabolic subgroup.  By the Langlands
decomposition, this parabolic subgroup has a reductive part and a
unipotent part, and the group element is the product of a bounded
element of the reductive part and an exponentially large unipotent
element.  Lemmas which conjugate unipotent elements by reductive
elements will then suffice to fill the resulting $\omega$-triangles.

For $\SL(4;\Z)$, we can say a little more.  One of the main advantages
of using $\SL(p;\Z)$ here is that when $p$ is large, it contains many
solvable subgroups (the $H_{S,T}$'s defined in
Sec.~\ref{sec:normalForms}) with quadratic Dehn functions and large
intersections -- this is one thing allowing us to prove, for example
Lemma~\ref{lem:shortEquiv}.  This gets more difficult for small $p$ because
the solvable groups and their intersections get smaller.

For example, when $p\ge 6$, Lemma~\ref{lem:shortEquiv} can be proved
in a few lines: Let $\gamma_{S,T}$ be a geodesic connecting $I$ and
$e_{1,6}(x)$ in $H_{S,T}$.  As long as $\#S\ge 2$ or $\#T\ge 2$, this
has length $\sim \log |x|$.  We can then construct a homotopy from,
say, $\gamma_{\{1,2,3\},\{6\}}$ to $\gamma_{\{1,3,4\},\{6\}}$ which
goes through the stages
$$\gamma_{\{1,2,3\},\{6\}}\to \gamma_{\{1\},\{5,6\}}\to
\gamma_{\{1,3,4\},\{6\}}.$$ In the first step, we use the fact that
both curves lie in $H_{\{1,2,3\},\{5,6\}}$, which has quadratic Dehn
function; likewise, in the second step, we use the fact that both
curves lie in $H_{\{1,3,4\},\{5,6\}}$.  The full lemma can be proved
in the same way.  When $p=5$, however, the lemma is more difficult to
prove, because the overlaps between solvable subgroups are smaller,
and when $p=4$, the lemma is not known.  In fact,
Lemma~\ref{lem:shortEquiv} is the main obstacle to proving a
polynomial bound on the Dehn function of $\SL(4;\Z)$.  In unpublished
work, I have reduced the problem of bounding the Dehn function of the
whole group by a polynomial to the problem of proving that
$\delta_{\SL(4;\Z)}(\short{e}_{ij}(x),\short{e}_{ij;S}(x))$ is bounded
by a polynomial in the length of $\short{e}_{ij}(x)$.

Similarly, one may ask about higher-order filling inequalities in
arithmetic groups.  These filling inequalities generalize the Dehn
function, but instead of bounding the area of a disc filling a curve
$\gamma$ of a given length, they bound the $(k+1)$-volume of a chain
filling a cycle of a given $k$-volume.  Gromov % check this
stated a generalization of Conj.\ref{conj:mainConj} to this situation
\begin{conj}\label{conj:higherConj}
  If $\Gamma$ is an irreducible lattice in a symmetric space with $\R$-rank at
  least $k+2$, then any $k$-cycle of volume $V$ has a filling by a
  $k$-chain of volume polynomial in $V$.
\end{conj}
Bestvina, Eskin, and Wortman \cite{BestEskWort} have made partial
progress toward a more general conjecture stated in terms of volume
distortion in lattices.  

\bibliography{slnz}
\end{document}